\documentclass[12pt,reqno]{amsart}
\usepackage[margin=1in]{geometry}
\makeatletter
\def\section{\@startsection{section}{1}%
	\z@{.7\linespacing\@plus\linespacing}{.5\linespacing}%
	{\bfseries%\normalfont\scshape
		\centering
}}
\def\@secnumfont{\bfseries}
\makeatother
\usepackage{amsmath,amssymb,amsthm,graphicx,amsxtra, setspace}
\usepackage[utf8]{inputenc}
\usepackage{mathrsfs}
\usepackage{alltt}
\usepackage{relsize}
\usepackage{hyperref}
\usepackage{aliascnt}
\usepackage{tikz}
\usepackage{mathtools}
\usepackage{multicol}
\usepackage{upgreek}
\usepackage{graphicx,type1cm,eso-pic,color}
\allowdisplaybreaks

\usepackage{pdfrender,xcolor}

\colorlet{darkblue}{blue!50!black}

\hypersetup{
	colorlinks,%
	citecolor=blue,%
	filecolor=red,%
	linkcolor=darkblue,%
	urlcolor=blue,%
	pdfnewwindow=true,%
	pdfstartview={FitH}
}

\newtheorem{theorem}{Theorem}[section]
\newtheorem{lemma}[theorem]{Lemma}

\newtheorem{remark}[theorem]{Remark}

\def\L{\mathrm{L}}

\def\X{\mathrm{X}}

\def\C{\mathrm{C}}

\def\D{\mathrm{D}}

\def\Y{\mathrm{Y}}

\def\N{\mathbb{N}}
\def\I{\mathrm{I}}

\def\H{\mathrm{H}}

\newcommand{\R}{\mathbb{R}}

\let\originalleft\left
\let\originalright\right
\renewcommand{\left}{\mathopen{}\mathclose\bgroup\originalleft}
\renewcommand{\right}{\aftergroup\egroup\originalright}

\newcommand{\vertiii}[1]{{\left\vert\kern-0.25ex\left\vert\kern-0.25ex\left\vert #1 
		\right\vert\kern-0.25ex\right\vert\kern-0.25ex\right\vert}}

\newcommand{\Addresses}{{% additional braces for segregating \footnotesize
		\footnote{
			\noindent \textsuperscript{1,2}Department of Mathematics, Indian Institute of Technology Roorkee-IIT Roorkee,
			Haridwar Highway, Roorkee, Uttarakhand 247667, INDIA.\par\nopagebreak
			\noindent  \textit{e-mail:} \texttt{Manil T. Mohan: maniltmohan@ma.iitr.ac.in, maniltmohan@gmail.com.}
			
			\textit{e-mail:} \texttt{Shri Lal Raghudev Ram Singh: raghudevram$\_$s@ma.iitr.ac.in.}
			
			\noindent \textsuperscript{*}Corresponding author.

			\textit{Key words:} generalized Korteweg-de Vries-Burgers-Huxley equation; monotonicity; Minty-Browder theorem; stabilization; boundary control.
			
			Mathematics Subject Classification (2020): 93D15, 93D23, 34H05, 35K51.

}}}

\begin{document}

	\title[Boundary control of generalized KdV-Burgers-Huxley equation]{Boundary control of generalized Korteweg-de Vries-Burgers-Huxley equation: Well-Posedness, Stabilization and  Numerical Studies\Addresses}
	\author[M. T. Mohan and S. L. Raghudev Ram Singh]
	{Manil T. Mohan\textsuperscript{1*} and Shri Lal Raghudev Ram Singh\textsuperscript{2}}

	\maketitle
	
	\begin{abstract}
		A boundary control problem for  the following generalized Korteweg-de Vries-Burgers-Huxley equation: $$u_t=\nu u_{xx}-\mu u_{xxx}-\alpha u^{\delta}u_x+\beta u(1-u^{\delta})(u^{\delta}-\gamma), \ x\in[0,1], \ t>0,$$ where $\nu,\mu,\alpha,\beta>0,$ $\delta\in[1,\infty)$, $\gamma\in(0,1)$  subject to Neumann boundary conditions is considered in this work. We first establish the well-posedness of the Neumann boundary value problem by an application of monotonicity arguments, the Hartman-Stampacchia theorem,  the Minty-Browder theorem,  and the Crandall-Liggett theorem.  The additional difficulties caused by the  third order linear term is successfully handled by proving a proper version of  the Minty-Browder theorem.  By using suitable  feedback boundary controls,  we demonstrate $\mathrm{L}^2$- and $\mathrm{H}^1$-stability properties of the closed-loop system for sufficiently large $\nu>0$. The analytical conclusions from this work are supported and validated by numerical investigations.
	\end{abstract}

	\section{Introduction}\label{sec1}\setcounter{equation}{0}  
	One of the most basic nonlinear mathematical models that illustrates the characteristics of both dispersion and dissipation is the Korteweg-de Vries-Burgers (KdVB) equation (\cite{ABMK,LBEC,ECJM,SMEC}). In this work, we consider a generalized Korteweg-de Vries-Burgers-Huxley equation which has the characteristics of \emph{dispersion, dissipation, convection and reaction effects}. 
	\subsection{The model} 
	One type of partial differential equation (PDE) for a shallow water wave with unidirectional motion is the \emph{Korteweg-de Vries (KdV) equation} (\cite{JLBRS}), 
	$$u_t=-\mu u_{xxx}-\alpha uu_x,\ \mu,\alpha>0,$$
	which is also a common representative of \emph{nonlinear dispersive equations}.  The \emph{viscous Burgers equation} (\cite{MK}), 
	$$u_t=\nu u_{xx}-\alpha uu_x, \ \nu,\alpha>0,$$ 
	a common representative of \emph{nonlinear dissipative equations} or \emph{convection-diffusion equation}, is used in a number of applied mathematics fields, including traffic flow, fluid mechanics, nonlinear acoustics, and gas dynamics. 	The case  of pure dispersion and pure dissipation is quite uncommon in many real-world physical problems. The \emph{Korteweg-de Vries-Burgers (KdVB)} equation is formed when a diffusion term is added to the KdV equation. The KdVB equations are given by (\cite{ABMK})
	$$u_t=\nu u_{xx}-\mu u_{xxx}-\alpha uu_x, \ \nu,\mu,\alpha>0.$$ It is discovered that this equation also represents a wide range of other physical phenomena, including lattice waves, the propagation of ion-acoustic waves in cold plasma, and pressure waves in liquid-gas bubble mixtures, etc. Actually, in addition to dispersion, several of these phenomena also show dissipative effects, which makes the so-called KdVB equation a better model (\cite{BCAG}).
		
		A nonlinear PDE that explains the interplay between reaction mechanisms, convection effects, and diffusion transports is the \emph{Burgers-Huxley equation} (\cite{XYW}):
	$$u_t=\nu u_{xx}-\alpha uu_x+\beta u(1-u)(u-\gamma),\ \nu,\mu,\alpha,\beta>0,\ \gamma\in(0,1).$$
	The \emph{Korteweg-de Vries-Burgers-Huxley (KdVBH)} equation is obtained by adding a reaction term similar to the one found in the Burgers-Huxley equation. We examine the following generalized version of KdVBH equation in this article:
	\begin{align}\label{1.1}
		u_t=\nu u_{xx}-\mu u_{xxx}-\alpha u^{\delta}u_x+\beta u(1-u^{\delta})(u^{\delta}-\gamma), \ x\in(0,1), \ t>0, \end{align}
where $\nu,\mu,\alpha,\beta>0,\ \gamma\in(0,1), \delta\in[1,\infty),$  with the initial condition
\begin{align}\label{1.2}
	u(x,0)=u_0(x),\ x\in[0,1]. 
\end{align}
	We call $\nu>0$ as the \emph{dissipation coefficient}, $\alpha>0$ as the \emph{convection coefficient} and $\mu>0$ as the \emph{dispersion coefficient.} The equation \eqref{1.1} can be considered as an example of a \emph{convection-diffusion-reaction equation of dispersive type.}

	Many authors have actively worked on the KdV and KdVB equations from a variety of angles (\cite{JLBSMS,EC,FLGP,TT}, etc., and references therein).  The aim of this work is to  analyze a stabilization problem by using a  boundary feedback control  for the system \eqref{1.1}-\eqref{1.2}. Similar problems for KdVB equation have been considered in the works \cite{ABMK,WJJ1}. 
	
	\subsection{The boundary control}
	Control problems for partial differential equations have been a focus of intense research for the past several decades. Generally speaking, there are two types of control procedures for partial differential equations: boundary control and distributed control. It is difficult for engineers to implement distributed control, which calls for actuators to be positioned at every point in the spatial domain (\cite{KWEF}). Boundary control is simple and inexpensive to implement; it only requires applying actuators on the edge of the spatial region.
	
	Let us now discuss some of the boundary control problems discussed in the literature for KdVB equation. For the forced KdVB equation, the author in \cite{BYZ} demonstrated that if the external forcing is time-periodic with small amplitude, then the problem also admits a unique time-periodic solution with the same period, and the solution is stable.  The boundary stabilization problem of KdVB equation is studied in \cite{ABMK,WJJ1}.  The authors of  \cite{ABMK,WJJ1} demonstrated well-posedness, global exponential stability in $\L^2$, global asymptotic  stability in $\H^1$, and semi-global exponential stability in $\L^2$ of  KdVB equation.   Nonlinear boundary stabilization for a generalized KdVB equation is examined in \cite{NSRA,NSRA1,NSAE}, etc.  By proposing an another class of adaptive controls, the authors in \cite{XDWC} established  well-posedness as well as  the $\L^2$-global exponential stability of the solutions of KdV and KdVB equations. By using suitable boundary controls, the authors in \cite{CJ,CJBY}  proved that the solutions of the KdVB equation  globally exist and globally exponentially tend to zero as $t\to\infty$ 	in a subspace of $\H^s$ for $s\in[0,3]$.  The authors in \cite{LBEC} developed two approaches for the stabilization of nonlinear KdV equation with boundary time-delay feedback. For more interesting problems on the control and stabilization of KdV and KdVB equations, we refer the interested readers to \cite{MA1,ECJM,EC,MoC,BCAG,FAG,FAG1,KWEF,TOAB,LR,LRBY,RS,SX}, etc., and references therein.

	The main aim of this article is to study the well-posedess of a boundary control problem for \eqref{1.1}-\eqref{1.2} and establish a feedback stabilization result. 	We use two different controls for $\delta=1$ and $\delta=2$. Motivated from \cite{ABMK}, keeping feedback stabilization in mind,  the  equation \eqref{1.1}-\eqref{1.2}  is associated with the following boundary conditions for $t\geq 0$: 
	
	For $\delta=1$, we consider  the control 
	\begin{equation}\label{1.3}
		\left\{
		\begin{aligned}
			u(0,t)&=0,\\ u_x(1,t)&=-\frac{1}{\nu}\left(\eta+\frac{\alpha^2}{\eta(\delta+2)^2}u^{2\delta}(1,t)\right)u(1,t)=-g_1(u(1,t)),\\ u_{xx}(1,t)&=\frac{\delta}{\nu^2}\left(\eta+\frac{\alpha^2}{\eta(\delta+2)^2}u^{2\delta}(1,t)\right)^2u(1,t)=g_2(u(1,t)), 
		\end{aligned}
		\right.
	\end{equation}
	where $\eta>0$, $g_1,g_2\in \C(\R)$ are nondecreasing functions with $g_1(0)=g_2(0)=0,$ and $g_2(k)k=\delta(g_1(k))^2$  representing the nonlinear flux feedback controls. 	 
	
	For $\delta=2$, we consider the control 
	\begin{equation}\label{1p3}
		\left\{
		\begin{aligned}
			u(0,t)&=0,\\ u_x(1,t)&=-\frac{1}{\nu}\left(\eta+\frac{\alpha}{(\delta+2)}u^{\delta}(1,t)\right)u(1,t)=-g_1(u(1,t)),\\ u_{xx}(1,t)&=\frac{\delta}{\nu^2}\left(\eta+\frac{\alpha}{(\delta+2)}u^{\delta}(1,t)\right)^2u(1,t)=g_2(u(1,t)), 
		\end{aligned}
		\right.
	\end{equation}
	where $\eta>0$.

	For all $u\in\H^3(0,1)$ satisfying \eqref{1.3}, an integration by parts yields  
	\begin{align}\label{1.4}
		(u_{xxx},u)&=-(u_{xx},u_x)+u_{xx}(1)u(1)-u_{xx}(0)u(0)\nonumber\\&=-\frac{1}{2}u_x^2(1)+u_{xx}(1)u(1)=-\frac{1}{2}(g_1(u(1)))^2+g_2(u(1))u(1)=\left(\delta-\frac{1}{2}\right)(g_1(u(1)))^2. 
	\end{align}
Similarly, 	 for all $u,v\in\H^3(0,1)$ satisfying \eqref{1.3} and for some $0<\theta<1$, we obtain 
	\begin{align}\label{1.5}
	&	((u-v)_{xxx},u-v)\nonumber\\&=-\frac{1}{2}(u_x(1)-v_x(1))^2+(u_{xx}(1)-v_{xx}(1))(u(1)-v(1))\nonumber\\&=-\frac{1}{2}\left(g_1(u(1))-g_1(v(1))\right)^2+\left(g_2(u(1))-g_2(v(1))\right)(u(1)-v(1))\nonumber\\&=-\frac{1}{2\nu^2}(u(1)-v(1))^2\bigg\{\eta^2+\frac{2\eta\alpha(\delta+1)}{(\delta+2)}(\theta u(1)+(1-\theta)v(1))^{\delta}\nonumber\\&\quad+\frac{\alpha^2(\delta+1)^2}{(\delta+2)^2}(\theta u(1)+(1-\theta)v(1))^{2\delta}\bigg\}+\frac{\delta}{\nu^2}(u(1)-v(1))^2\bigg\{\eta^2\nonumber\\&\quad+\frac{2\eta\alpha(\delta+1)}{(\delta+2)}(\theta u(1)+(1-\theta)v(1))^{\delta}+\frac{\alpha^2(2\delta+1)}{(\delta+2)^2}(\theta u(1)+(1-\theta)v(1))^{2\delta}\bigg\}\nonumber\\&=\frac{1}{\nu^2}(u(1)-v(1))^2\bigg\{\left(\delta-\frac{1}{2}\right)\left[\eta^2+\frac{2\eta\alpha(\delta+1)}{(\delta+2)}(\theta u(1)+(1-\theta)v(1))^{\delta}\right]\nonumber\\&\quad+\frac{\alpha^2(3\delta^2-1)}{2(\delta+2)^2}(\theta u(1)+(1-\theta)v(1))^{2\delta}\bigg\}.%\nonumber\\&=\frac{1}{2}(u(1)-v(1))^2\bigg\{\left(\eta+\frac{\alpha(\delta+1)}{(\delta+2)}(\theta u(1)+(1-\theta)v(1))^{\delta}\right)^2-\frac{2\alpha^2\delta^2}{(\delta+2)^2}(\theta u(1)+(1-\theta)v(1))^{2\delta}\bigg\}.
	\end{align}
	It can be seen that $((u-v)_{xxx},u-v)\geq 0$. A similar calculation can be performed for $\delta=1$ also. %provided $\eta>0$ for $1\leq \delta\leq\sqrt{2}+1$ and $\eta>\frac{((\sqrt{2}-1)\delta-1)}{\delta+2}\alpha(\theta u(1)+(1-\theta)v(1))^{\delta}$ (sufficiently large) for $\sqrt{2}+1<\delta<\infty$.

	\begin{remark}\label{rem1.1}
		For the case of any $\nu>0$, we are providing the controls for $\delta=1$ and $\delta=2$ in \eqref{1.3} and \eqref{1p3}, respectively, keeping in mind that one may extend to other values of $\delta$ also in the future. For $\delta\in[1,\infty)$ and  large values of $\nu$, one can  use the control given in  \eqref{256} below. For  small values of $\nu>0,$ we restrict ourselves to $\delta=1,2$ due to the unavailability of a well-posedness result for other values of $\delta$ (see Theorem \ref{thm2.5} below). 
		
		Moreover, a control like  the following also stabilizes the system \eqref{1.1}-\eqref{1.2} (\cite{WJJ1}):
		\begin{align}\label{17}
			u(0,t)=0,\ u_x(1,t)=0,\ u_{xx}(1,t)=g(u(1,t)).
		\end{align}
In this case, the condition \eqref{1.5} reduces to $\left(g(u(1))-g(v(1))\right)(u(1)-v(1))\geq 0$, that is, $g(\cdot)$ needs to be a monotonicially increasing function. 	But the rate of convergence is slower than the controls considered in \eqref{1p3} and \eqref{1.3}. Examples of such $g(\cdot)$ are 
	\begin{align}\label{18}
	 \frac{1}{\mu}\left(\eta+\frac{\alpha^2}{\eta(\delta+2)^2}u^{2\delta}(1,t)\right)u(1,t) \ (\text{for }\ \delta=1) \ \text{ and }\ \frac{\eta}{\mu}u(1,t) \ (\text{for }\ \delta=2),
		\end{align}
for some $\eta>0$. 
%	\textcolor{red}{	One can consider the following feedback control law also:
%		\begin{align}
%			\left\{
%			\begin{aligned}
%				u(0,t)&=0,\\ u_x(1,t)&=-\frac{\eta}{\nu}u(1,t),\\ u_{xx}(1,t)&=\eta u(1,t)+\frac{\eta^2}{2\nu^2}u(1,t)-\frac{\alpha}{\mu(\delta+2)}u^{\delta+1}(1,t),
%			\end{aligned}
%			\right.
%		\end{align}
%		where $\eta>0$.}
		\end{remark}

			Three different adaptive control laws  are designed in \cite{NSAE1} to show the  $\L^2$-global exponential stability  for   the adaptive control problem of a forced generalized KdVB equation  	when either the kinematic viscosity and/or the dynamic viscosity are unknown. But the controls used in \cite{NSAE1}  may not be useful for our case (even for non-adaptive control problem) as the condition \eqref{1.5} may not be satisfied by such controls (see Remark \ref{rem1.1}). The monotonicity condition \eqref{1.5} is crucial in establishing the well-posedness of the boundary control problem for the generalized KdVBH equation \eqref{1.1}-\eqref{1.2}.

		\subsection{Difficulties, approaches  and novelties} The major difficulty of the work lies in establishing the well-posedness of the  generalized KdV-Burgers-Huxley equation \eqref{1.1}-\eqref{1.2} with the controls \eqref{1p3} and \eqref{1.3} for $\delta=1$ and $\delta=2$, respectively. As the problem \eqref{1.1} is of the third order dispersive type,  there are limitations in using the standard Minty-Browder result. The classical result on Minty-Browder Theorem states that  any monotone, hemicontinuous and strongly coercive operator $T : \X\to\X^*$,  where $\X$ is a reflexive Banach space, is onto, that is, $\mathrm{Range}(T)=\X^*$ (\cite[Theorem 3.3.1]{GDJM}).  By properly defining hemicontinuity and coercivity in our context (see \eqref{2k6} and \eqref{2p7} below) and applying the infinite-dimensional version of the Hartman-Stampacchia theorem (\cite[Theorem 1.4, Chapter III]{DKGS}), we show the same result for $T : \X\to\Y^*,$ that is, $\mathrm{Range}(T)=\Y^*$ (Theorem \ref{thm2.2}), where $\Y$ is also a reflexive Banach space such that the embedding $\X\subset\Y$ is dense (Lemma \ref{lem2p1}). We hope that the abstract result obtained in Theorem \ref{thm2.2} can be used to prove the well-posedness of various systems of the type \eqref{1.1}. 
		
		Another major difficulty is the  restriction of $\delta\in\{1,2\}$ for any $\nu>0$. This restriction is due to the choice of our controls given in \eqref{1.3}  and \eqref{1p3} and the lack of existence and uniqueness results for the other values of $\delta$ (see \eqref{244} in the proof of Theorem \ref{thm2.5}).  For $\delta\in\{1,2\}$, we use monotonicity arguments (Lemma \ref{lem2.1}),  the Minty-Browder theorem (Theorem \ref{thm2.2}), and the Crandall-Liggett theorem to prove  the existence and uniqueness of strong solutions for a cutoff problem (see \eqref{225} below). Then, by using uniform energy estimates, we show the global solvability of   the problem \eqref{1.1}-\eqref{1.2} with the controls \eqref{1p3} and \eqref{1.3} for $\delta=1$ and $\delta=2$, respectively (Theorem \ref{thm2.5}).  The coercivity results (see \eqref{215} and \eqref{2k12} below) immediately help us to obtain the $\L^2$-exponential stabilization of the  problem \eqref{1.1}-\eqref{1.2} under the assumption $\nu>\frac{\beta}{4}(1-\gamma)^2$ (Theorem \ref{thm2.3}). But under a restrictive  assumption on $\nu$, that is,  $\nu>\frac{\alpha^2}{2\beta}(\delta+2)^{-\frac{2(\delta+1)}{\delta+2}}$, we observe that the convective term can be handled by using the diffusion as well as reaction terms for any $\delta\in[1,\infty)$  and the well-posedness can be established.  For sufficiently large $\nu$ (see \eqref{2p55} below),  we are also able to obtain $\L^2$, $\H^1$ and pointwise exponential stabilization results (Theorems \ref{thm2.4}, \ref{thm3.8} and \ref{thm3.10}).

		\subsection{Organization of the paper} The rest of the paper is organized as follows: The next section deals with the well-posedness results for the problem \eqref{1.1}-\eqref{1.2} with the controls \eqref{1p3} and \eqref{1.3} for $\delta=1$ and $\delta=2$, respectively. We show the well-posedness of the Neumann boundary value problem (Theorem \ref{thm2.5}) by an application of monotonicity arguments,  the Minty-Browder theorem (Theorem \ref{thm2.2}), the Hartman-Stampacchia theorem \cite[Theorem 1.4, Chapter III]{DKGS}  and the Crandall-Liggett theorem (\cite[Theorems 5.1, 5.2]{JAW}).  In section \ref{sec3}, the $\L^2$-exponential stabilization result for the above problem with $\nu>\frac{\beta}{4}(1-\gamma)^2$ is established  (Theorem \ref{thm2.3}). Under further assumption on $\nu$  (see \eqref{2p55} below), $\L^2$, $\H^1$ and pointwise exponential stabilization results for the problem \eqref{1.1}-\eqref{1.2} with the control \eqref{256} is obtained in Theorems \ref{thm2.4}, \ref{thm3.8} and \ref{thm3.10}, respectively.  Numerical investigations in section \ref{sec4} validate and corroborate the analytical conclusions drawn in section \ref{sec2}. An example is also provided to show that that controls given in \eqref{1.3} and \eqref{1p3} converge much faster than the one given in \eqref{18}.

	\section{Well-posedness and Stabilization}\label{sec2}\setcounter{equation}{0}  
 In order to formulate the problem \eqref{1.1}-\eqref{1.3} as an abstract initial value problem,	we consider the Hilbert space $\mathrm{H}=\mathrm{L}^2(0,1)$ and the operator $\mathscr{A}:\D(\mathscr{A})\subset\H\to\H$  as
\begin{align}\label{2.1}
	\mathscr{A}(v):=	-\nu v_{xx}+\mu v_{xxx}+\frac{\alpha}{\delta+1} (v^{\delta+1})_x-\beta v(1-v^{\delta})(v^{\delta}-\gamma),
\end{align}
and the domain 
\begin{align}\label{2.2}
	\D(\mathscr{A}):=\left\{u\in\H^3(0,1):u(0)=0,\ u_x(1)=-g_1(u(1)),\ u_{xx}(1)=g_2(u(1))\right\}.
\end{align}
%Let us define the operators, $\mathscr{A}_1:\D(\mathscr{A}_1)\subset\H\to\H$ by $$\mathscr{A}_1(v)=-\nu_1v_{xx}+\mu v_{xxx}$$ with the domain $\D(\mathscr{A}_1)=\D(\mathscr{A}),$ and $\mathscr{A}_2:\D(\mathscr{A}_2)\subset\H\to\H$  by $$\mathscr{A}_2(v)=-\nu_2v_{xx}+\frac{\alpha}{\delta+1} (v^{\delta+1})_x-\beta v(1-v^{\delta})(v^{\delta}-\gamma)$$ with the domain $\D(\mathscr{A}_2)=\left\{u\in\H^2(0,1):u(0)=0, u_x(1)=-g_1(u(1))\right\}$, where $\nu_1,\nu_2>0$ are such that $\nu=\nu_1+\nu_2$. Note that $\mathscr{A}(v)=\mathscr{A}_1(v)+\mathscr{A}_2(v)$ for all $v\in\D(\mathscr{A})$.
With the above notations, the problem \eqref{1.1}-\eqref{1.3} can be reformulated as 
\begin{equation}\label{2.3}
	\left\{
	\begin{aligned}
		\frac{du(t)}{dt}+\mathscr{A}(u)(t)&=0,\ t>0,\\
		u(0)&=u_0. 
	\end{aligned}\right.
\end{equation}
	Performing an integration by parts in \eqref{2.1},  for each $v\in\D(\mathscr{A})$, we have
\begin{align}\label{2.6}
	(\mathscr{A}(v),w)&=\nu(v_x,w_x)+\nu g_1(v(1))w(1)-\mu(v_{xx},w_x)+\mu g_2(v(1))w(1)\nonumber\\&\quad+\frac{\alpha}{\delta+1}\left[v^{\delta+1}(1)w(1)-(v^{\delta+1},w_x)\right]-\beta (v(1-v^{\delta})(v^{\delta}-\gamma),w),
\end{align}
for all $$w\in \mathrm{X}:=\left\{u\in\H^2(0,1):u(0)=0,u_x(1)=-g_1(u(1))\right\}.$$ Let us define $\Y:=\left\{u\in\H^1(0,1):u(0)=0\right\}$. Clearly $\X\subset\Y$ and the embedding is continuous.  
Let us first prove the following technical result: 
%Moreover, we have 
%\begin{align}
%	(\mathscr{A}(v),-v_{xx})&=\nu\|v_{xx}\|_{\L^2}^2-\frac{\mu}{2}g_2^2(v(1))+\frac{\mu}{2}v_{xx}^2(0)-\alpha(v^{\delta}v_x,v_{xx})+\beta(1+\gamma)(v^{\delta+1},v_{xx})\nonumber\\&\quad+\beta\gamma\|v_x\|_{\L^2}^2+\beta\gamma g_1(v(1))v(1)+\beta(2\delta+1)\|v^{\delta}v_x\|_{\L^2}^2+\beta g(v(1))v^{2\delta+1}(1), 
%\end{align}
%for all $v\in\D(\mathscr{A})$. 
%We also define $$\Y:=\left\{u\in\H^1(0,1):u(0)=0\right\}.$$  It is immediate that $\X\subset\Y$ and \textcolor{red}{the embedding is dense.  Since the embedding $\H^2(0,1)\subset\H^1(0,1)$ is dense,   for any $v\in\H^1(0,1)$, there exists a sequence  $\{v_n\}_{n\in\N}\subset\H^2(0,1)$ such that $\|v_n-v\|_{\H^1}\to 0$ as $n\to\infty$.  Therefore, for any $v\in\Y$, we can choose $v_n\in\X$ such that $v_n(0)=0$ and $v_{n,x}(1)=-g_1(v_n(1))$ (since $\H^2(0,1)\subset\C^1([0,1])$) such that $\|v_n-v\|_{\H^1}\to 0$ as $n\to\infty$.}
\begin{lemma}\label{lem2p1}
	The space $\X$ is reflexive and is a dense subspace of $\Y$. 
\end{lemma}
\begin{proof}
	Remember that closed subspaces of reflexive Banach spaces are reflexive (\cite[Theorem 5.9]{SK}). As $\H^2(0,1)$ is a Hilbert space, in order to prove $\X$ is reflexive,  it is enough to show that $\X$ is closed.
	\vskip 0.1 cm
	\noindent
	\textbf{Claim:} \emph{$\X$ is a closed subspace of $\H^2(0,1)$.} Let $u_n\in\X$ bs such that $\|u_n-u\|_{\H^2}\to 0$ as $n\to\infty$. Clearly $u\in\H^2(0,1)$, $u_n(0)=0$ and $u_{n,x}(1)=-g_1(u_n(1))$. We need to show that $u(0)=0$ and $u_{x}(1)=-g_1(u_(1))$. The convergence $\|u_n-u\|_{\H^2}\to 0$ implies $\|u_n-u\|_{\H^1}\to 0$, and since $\H^1(0,1)\hookrightarrow\C([0,1])$, we deuce $u_n(x)\to u(x)$ for all $x\in[0,1]$ and $u\in\C([0,1])$ with $u(0)=0$. Moreover, it is immediate that $u_n(1)\to u(1)$ and since $g_1$ is continuous function, we get $g_1(u_n(1))\to g_1(u(1))$ as $n\to\infty$. Since $\H^1(0,1)\hookrightarrow\C^1([0,1])$, one can easily deduce that $u_{n,x}(x)\to u_x(x)$ for all $x\in[0,1]$ and by using the uniqueness of limit, we deduce $u_x(1)=-g_1(u(1))$. Therefore,  $u\in\X$ and hence $\X$ is a closed subspace of $\H^2(0,1)$. 
	
	\vskip 0.1 cm
	\noindent
	\textbf{Claim:} \emph{$\X$ is a dense subspace of $\Y$.}
	Fix any $v\in\Y$, that is, $v\in\H^1(0,1)$ and $v(0)=0$. As a subspace of $\H^1(0,1)$, the space $\Y$ inherits the $\H^1$-norm, but as $v(0)=0$, the $\Y$-norm is equivalent to $\|v_x\|_{\L^2}$ for all $v\in\Y$. It is just a simple application of  the Poincar\'e inequality (see Lemma \ref{lem21} below).   Since the embedding $\H^2(0,1)\subset\H^1(0,1)$ is dense,   there exists a sequence  $\{v_m\}_{m\in\N}\subset\H^2(0,1)$ such that $\|v_m-v\|_{\H^1}\to 0$ as $m\to\infty$.  Since $\H^2(0,1)\hookrightarrow\C^1([0,1])$,  we choose the sequence $\{v_m\}_{m\in\N}$ in such a way that $v_{m,x}(1)\to 0$ as $m\to\infty$. By defining $w_m(x)=v_m(x)-v_m(0)$, for all $x\in[0,1]$, we find $w_m(0)=0$ and $\|w_m-v\|_{\Y}=\|(v_{m}-v_m(0)-v)_x\|_{\L^2}=\|v_{m,x}-v_x\|_{\L^2}\to 0$ as $m\to\infty$. Therefore, the space $\{v\in\H^2(0,1):v(0)=0\}$ is dense in $\Y$. 
	
	Let us now show that $\X$ is dense in $\Y$.  We consider for $m\geq 2$, $\varphi_m(x)=x(1-x)^m$. Then, it is clear that $\varphi_m(0)=\varphi_m(1)=\varphi_{m,x}(1)=0$ and $\varphi_{m,x}(0)=1$. Then, by using properties of the beta function, it is immediate that 
	\begin{align*}
		\|\varphi_m\|_{\Y}^2=\|\varphi_{m,x}\|_{\L^2}^2=\frac{m}{(2m+1)(2m-1)}\to 0\ \text{ as }\ m\to\infty. 
	\end{align*}
Since $m\geq 2$, it is clear that $\varphi_m\in\H^2(0,1)$. Let us  define 
\begin{align*}
	z_m(x):=w_m(x)+(w_{m,x}(1)+g_1(w_m(1)))\varphi_m(1-x).
	\end{align*}
	Then $z_m\in\H^2(0,1)$ for all $m\in\N$ with $z_m(0)=0$ and $z_{m,x}(1)=-g_1(w_m(1))=-g_1(z_m(1))$, so that $z_m\in\X$. Since $\H^1(0,1)\hookrightarrow\C([0,1])$, the convergence $\|w_m-v\|_{\Y}\to 0$ as $m\to\infty$ implies that $w_m(x)\to v(x)$ as $m\to\infty$ for all $x\in[0,1]$. Since $g_1(\cdot)$ is a continuous function $g_1(w_m(1))\to g_1(v(1))$ as $m\to\infty$. Therefore $|g_1(w_m(1))|\leq K_1$ for all $m\in\N$. Note that  $w_{m,x}(1)=v_{m,x}(1)\to 0$ as $m\to\infty$. Thus, it can be easily seen that 
	\begin{align*}
		\|z_m-v\|_{\Y}&=\|z_{m,x}-v_x\|_{\L^2}\leq \|w_{m,x}-v_x\|_{\L^2}+(|w_{m,x}(1)|+|g_1(w_m(1))|)\|\varphi_{m,x}\|_{\L^2}\nonumber\\&\to 0\ \text{ as }\ m\to\infty,
	\end{align*}
which completes the proof. 
\end{proof}

\begin{remark}
Since the continuous embedding  $\X\subset\Y$ is dense, from \cite[Lemma 2.2.27]{LGNSP}, we infer that the embedding $\Y^*\subset\X^*$ is continuous and the reflexivity of  $\X$ implies that $\Y^*$ is dense in $\X^*$. 
\end{remark}

\subsection{Abstract result}
 For a mapping $T:\X\to\Y^*$, the following result is a slight modification of the results presented in \cite[Corollary 2.1.3.]{VB}, \cite[Theorem 3.3.1]{GDJM}, etc., where the mapping  $T:\X\to\X^*$. Therefore, we are providing a proof here (Minty's theorem). 

\begin{theorem}\label{thm2.2}
	Let $\X$ and $\Y$ be reflexive Banach spaces such that the continuous embedding $\X\hookrightarrow\Y$ is dense. Let $T:\X\to\Y^*$  be monotone, that is, \begin{align}\label{2p5}\langle T(u)-T(v),u-v\rangle\geq 0\ \text{ for all }\ u,v\in\X,\end{align}\label{2k6} hemicontinuous, that is,  \begin{align}\label{2p6}\lim\limits_{\lambda\to 0}\langle T(u_1+\lambda u_2),v\rangle=\langle T(u_1),v\rangle,\ \text{ for all }\ u_1,u_2\in\X,\ \text{ and }\ v\in\Y,\end{align}  and  coercive, that is,  \begin{align}\label{2p7}\lim\limits_{\|u\|_{\Y\to\infty}}\frac{\langle T(u),u\rangle}{\|u\|_{\Y}}=\infty, \ \text{ for all }\ u\in\X, \end{align} operator. Then $T$ is onto, that is, $\mathrm{Range}(T)=\Y^*$.
\end{theorem}
\begin{proof}
	Let $f\in \Y^*$. We have for all $\X\ni u\neq 0$, 
	\begin{align}
		\langle T(u)-f,u\rangle \geq \|u\|_{\Y}\left[\frac{\langle T(u),u\rangle}{\|u\|_{\Y}}-\|f\|_{\Y^*}\right].
	\end{align}
	Since $T$ is coercive in the sense of \eqref{2p7}, for any $N>0$, there exists an $R>0$ such that $\frac{\langle T(u),u\rangle}{\|u\|_{\Y}}>N$ for all $\|u\|_{\Y}\geq R$. Therefore, for $N>\|f\|_{\Y^*}$, there exists an $R>0$ such that 
	\begin{align}\label{2p22}
		\langle T(u)-f,u\rangle >0 \ \text{ for all }\  \|u\|_{\Y}\geq R.
	\end{align}
	
	Let us define $\mathbb{K}:=\left\{u\in\X:\|u\|_{\Y}\leq R\right\}$. Then $\mathbb{K}$ is a closed and convex subset of $\X$.  We further define $\mathbb{K}_R=\mathbb{K}\cap\Sigma_R=\Sigma_R$, where $\Sigma_R:=\left\{u\in\X:\|u\|_{\X}\leq R\right\}$.  Then $\mathbb{K}_R$ is not only a closed and convex but also a bounded subset of $\X$.  As $\X$ is reflexive (Lemma \ref{lem2p1}), and $T$ is monotone and hemicontinuous,  we infer from \cite[Lemma 3.3.2]{GDJM} that $T$ is  demicontinuous.  The reflexivity of $\X$  and \cite[Remark 5.21]{VR} imply  that the demicontinuous operator $T: \mathbb{K}_R\to\Y^*$  is continuous on finite dimensional subspaces of $\mathbb{K}_R,$ that is, for any finite dimensional subspace $M\subset\X$, the restriction of $T$ to $\mathbb{K}_R\cap M$ is weakly continuous, namely, $T:\mathbb{K}_R\cap M\to\Y^*$ is weakly continuous. Together with this fact, the monotonicity of the operator $T:\mathbb{K}_R\to\Y^*$ and the Hartman-Stampacchia theorem (see \cite[Theorem 1.4, Chapter III]{DKGS}) yield the existence of a $u_R\in\mathbb{K}_R$ such that 
	\begin{align}\label{2p10}
		\langle T(u_R)-f, w-u_R\rangle\geq 0\ \text{ for all }\ w\in\mathbb{K}_R. 
	\end{align}
	
	The monotonicity as well as coercivity properties  of $T:\mathbb{K}\to\Y^*$   imply (see \eqref{2p5} and \eqref{2p7})
	\begin{align*}
		\frac{\langle(T(u)-f)-(T(u_*)-f),u-u_*\rangle}{\|u-u_*\|_{\Y}}\to\infty \ \text{ as }\ \|u\|_{\Y}\to\infty, \ u\in\mathbb{K},
	\end{align*}
	for some $u_*\in\mathbb{K}_R$. For instance, one can choose $u_*=0\in\mathbb{K}_R$. Choose, $K>\|T(u_*)-f\|_{\Y^*}$ and $R>\|u_*\|_{\Y}$ such that 
	\begin{align*}
		\langle(T(u)-f)-(T(u_*)-f),u-u_*\rangle\geq K\|u-u_*\|_{\Y}\ \text{ for }\ \|u\|_{\Y}\geq R, \ u\in\mathbb{K}. 
	\end{align*}
	Therefore, 
	\begin{align}\label{2a12}
			\langle(T(u)-f),u-u_*\rangle&\geq K\|u-u_*\|_{\Y}+	\langle(T(u_*)-f),u-u_*\rangle\nonumber\\&\geq K\|u-u_*\|_{\Y}-\|T(u_*)-f\|_{\Y^*}\|u-u_*\|_{\Y}\nonumber\\&\geq \left(K-\|T(u_*)-f\|_{\Y^*}\right)\left(\|u\|_{\Y}-\|u_*\|_{\Y}\right)>0\ \text{ for }\ \|u\|_{\Y}=R.
	\end{align}
	Now, let $u_R\in\mathbb{K}_R$ be a solution of \eqref{2p10}. Then 
	\begin{align*}
		\langle(T(u_R)-f),u_R-u_*\rangle=-\langle(T(u_R)-f),u_*-u_R\rangle\leq 0,
	\end{align*}
	and in view of \eqref{2a12}, we have $\|u_R\|_{\Y}\neq R$. Or in other words, $\|u_R\|_{\Y}< R$.  Combining  the above facts and applying \cite[Theorem 1.7, Chapter III]{DKGS}, we deduce 
	% Since $T$ is monotone and hemicontinuous  in the sense of \eqref{2p12}, an application of \cite[Theorem 3.2.5]{GDJM} yields 
	the existence of  a $u_0 \in \mathbb{K}$ such that
	\begin{align}\label{2p23}
		\langle T(u_0)-f,w-u_0\rangle \geq 0\ \text{ for all }\ w\in \mathbb{K}.
	\end{align}
		In particular, for $w=0$, we have $\langle T(u_0)-f,u_0\rangle \leq 0$, which by \eqref{2p22}  implies that $$u_0\in \left\{u\in\X:\|u\|_{\Y}< R\right\}.$$ Consequently taking $w=u_0+\lambda v$ in \eqref{2p23} with $\|v\|_{\Y}=1$ and $0<|\lambda|\leq R-\|u_0\|_{\Y}$, we obtain 
	\begin{align}\label{2p24}
		\langle T(u_0)-f,v\rangle = 0\ \text{ for all }\ v\in \{u\in\X:\|u\|_{\Y}=1\}. 
	\end{align}
	Since the embedding $\X\subset\Y$ is dense, for any $v\in\Y$ with $\|v\|_{\Y}=1$, we can find a sequence  $\{v_n\}_{n\in\N}\subset\X$ with $\|v_n\|_{\Y}=1$ such that  $\|v_n-v\|_{\Y}\to 0$ as $n\to\infty$ (Lemma \ref{lem2p1}). Therefore, from \eqref{2p24}, we infer 
	\begin{align*}
		\lim\limits_{n\to\infty}	\langle T(u_0)-f,v_n\rangle = 0\Rightarrow \langle T(u_0)-f,v\rangle = 0\ \text{ for all }\  v\in\{u\in\Y:\|u\|_{\Y}=1\},
	\end{align*}
	and hence $T(u_0)=f\in\Y^*.$
\end{proof}

\subsection{Existence and uniqueness}
We use Theorem \ref{thm2.2} to show the existence and uniqueness of solutions  for the problem \eqref{1.1}-\eqref{1.2} with the controls \eqref{1p3} and \eqref{1.3} for $\delta=1$ and $\delta=2$, respectively. 
We provide a cut-off function for the function $f(v)=\frac{1}{\delta+1}v^{\delta+1}$ in order to obtain monotonicity of the nonlinear map $\mathscr{A}$ defined in \eqref{2.6} due to the lack of global Lipschitz continuity. 
%Actually, we prove in general that the following findings hold for any $f\in \C^1(\R)$.   
%For $\rho\geq 1$, we define 
%\begin{align*}
%	f_{\rho}(y):=\left\{\begin{array}{ll}
%		f(\rho)+f'(\rho)(y-\rho)&\ \text{ if }\ y>\rho,\\ 
%		f(y)&\ \text{ if }\ |y|\leq\rho,\\
%		f(-\rho)+f'(-\rho)(y+\rho)&\ \text{ if }\ y<-\rho. 
%	\end{array}\right.
%\end{align*}
%Then one can show that $f_{\rho}\in\C^1(\R)$ is globally Lipschitz continuous with the Lipschitz constant, $$L_{\rho}=\sup_{x\in\R}|f_{\rho}'(x)|=\sup_{|x|\leq\rho}|f'(x)|.$$  

Let us define the modified nonlinearity $f_{\rho}:\Y\to\Y^*$ by 
\begin{align}\label{2q12}
	f_{\rho}(y):=\left\{\begin{array}{cc}f(y)&\ \text{ if }\ |y|\leq \rho,\\
	\left(\frac{\rho}{\|y\|_{\L^{\infty}}}\right)^{\delta+1}f(y)&\ \text{ if }\ |y|> \rho,\end{array}\right.
\end{align}
for $\rho\in\mathbb{N}$. Since each $y\in\Y$ is a continuous function over $[0,1]$, $\|y\|_{\L^{\infty}}=\max\limits_{x\in[0,1]}|y(x)|.$

Let us define the  nonlinear map $\mathscr{A}_{\rho}:\mathrm{X}\subset\Y\to\mathrm{Y}^*$ by 
\begin{align}\label{25}
	\langle\mathscr{A}_{\rho}(v),w\rangle &=\nu(v_x,w_x)+\nu g_1(v(1))w(1)-\mu(v_{xx},w_x)+\mu g_2(v(1))w(1)\nonumber\\&\quad+\alpha\left[f_{\rho}(v(1))w(1)-(f_{\rho}(v),w_x)\right]-\beta (v(1-v^{\delta})(v^{\delta}-\gamma),w),
\end{align} 
for all $v\in\mathrm{X}$ and $w\in\mathrm{Y}$. Therefore, for each $v\in\X$ and $w\in\Y$, we have 
\begin{align}\label{26}
	&	|\langle\mathscr{A}_{\rho}(v),w\rangle |\nonumber\\&\leq C \big[\nu(\|v_x\|_{\mathrm{L}^2}+|g_1(v(1))|)+\mu(\|v_{xx}\|_{\L^2}+|g_2(v(1))|)+\alpha(\|f_{\rho}(v)\|_{\L^2}+|f_{\rho}(v(1))|)\nonumber\\&\quad+\beta((1+\gamma)\|v\|_{\L^{2(\delta+1)}}^{\delta+1}+\gamma\|v\|_{\L^2}+\|v\|_{\L^{2(\delta+1)}}^{2\delta+1})\big]\|w\|_{\H^1},
\end{align}
where we have used the continuous Sobolev embedding  $\H^1(0,1)\hookrightarrow\C([0,1])\hookrightarrow\L^p(0,1)$ for all $1\leq p<\infty$. From \eqref{26}, it is immediate that the operator $\|\mathscr{A}_{\rho}(v)\|_{\Y^*}$ is bounded for each $v\in\X$.

\begin{remark}
	Note that if one takes $w=v$ in \eqref{26}, then we have  for all $v\in\X$:
	\begin{align}\label{217}
		\langle\mathscr{A}_{\rho}(v),v\rangle &=\nu\|v_x\|_{\L^2}^2+\nu g_1(v(1))v(1)-\frac{\mu}{2}(g_1(v(1)))^2+\frac{\mu}{2}(v_x(0))^2+\mu g_2(v(1))v(1)\nonumber\\&\quad+\alpha\left[f_{\rho}(v(1))v(1)-(f_{\rho}(v),v_x)\right]-\beta(1+\gamma)(v^{\delta+1},v)+\beta\gamma\|v\|_{\L^2}^2+\beta\|v\|_{\L^{2(\delta+1)}}^{2(\delta+1)}.
	\end{align}
		For $\delta=1$ and $|v|\leq\rho$, using the control given in \eqref{1.3}, we find
	\begin{align}\label{2k12}
		\langle\mathscr{A}_{\rho}(v),v\rangle &\geq \nu\|v_x\|_{\L^2}^2+\eta v^2(1)+\frac{\alpha^2}{\eta(\delta+2)^2}v^{2(\delta+1)}(1)+\frac{\mu}{2}(g_1(v(1)))^2+\frac{\alpha}{\delta+2}v^{\delta+2}(1)\nonumber\\&\quad-\frac{\beta}{4}(1-\gamma)^2\|v\|_{\L^2}^2\nonumber\\&\geq \nu\|v_x\|_{\L^2}^2+\frac{\eta}{2} v^2(1)+\frac{\alpha^2}{2\eta(\delta+2)^2}v^{2(\delta+1)}(1)+\frac{\mu}{2}(g_1(v(1)))^2-\frac{\beta}{4}(1-\gamma)^2\|v\|_{\L^2}^2\nonumber\\&\geq\left(\nu-\frac{\beta}{4}(1-\gamma)^2\right)\|v\|_{\L^2}^2\geq  0,
	\end{align}
	provided  $\nu\geq\frac{\beta}{4}(1-\gamma)^2$. 
	For $\delta=2$ and $|v|\leq\rho$, using the control given in \eqref{1p3}, we deduce 
	\begin{align}\label{215}
		\langle\mathscr{A}_{\rho}(v),v\rangle &\geq\nu\|v_x\|_{\L^2}^2+\eta v^2(1)+\frac{2\alpha}{(\delta+2)}v^{\delta+2}(1)+\frac{\mu}{2}(g_1(v(1)))^2-\frac{\beta}{4}(1-\gamma)^2\|v\|_{\L^2}^2\nonumber\\&\geq\left(\nu-\frac{\beta}{4}(1-\gamma)^2\right)\|v\|_{\L^2}^2\geq  0,
		\end{align}
		provided  $\nu\geq\frac{\beta}{4}(1-\gamma)^2$.
As $1>\frac{\rho}{\|v\|_{\L^{\infty}}}$, the same results hold true for $|v|>\rho$ also. 
	%For $|v|\leq\rho$, we find by an application of Poincar\'e's inequality that 
	%\begin{align}\label{29}
	%	\langle\mathscr{A}(v),v\rangle =	\langle\mathscr{A}_{\rho}(v),v\rangle &\geq \nu\|v_x\|_{\L^2}^2-\beta(1-\gamma)^2\|v\|_{\L^2}^2+\eta v^2(1)+\frac{2\alpha}{\delta+2}v^{\delta+2}(1)\nonumber\\&\quad+\frac{\mu}{2\nu^2}\left(\eta+\frac{\alpha}{\delta+2}v^{\delta}(1)\right)^2v^2(1)\nonumber\\&\geq(\nu-\beta(1-\gamma)^2)\|v\|_{\L^2}^2\geq  0,
	%\end{align}
	%if $\nu\geq\beta(1-\gamma)^2$, $\delta$ is even or $v(1)$ is positive. If $\delta$ is odd and $v(1)$ is negative, then one can consider the feedback control given in \eqref{} to obtain the required result. 
\end{remark}

The following result shows that the operator $\mathscr{B}_{\rho,\omega}:=\mathscr{A}_{\rho}+\omega\I:\X\subset\Y\to\Y^*$ is monotone, hemicontinuous and coercive for sufficiently large $\omega$ (in the sense of Theorem \ref{thm2.2}). 

\begin{lemma}\label{lem2.1}
	Let $g_1,g_2\in\C(\R)$ be non-decreasing functions, $f\in\C^1(\R)$, $\nu, \mu, \alpha,\beta,\rho>0$, $\delta\in\{1,2\}$, $\gamma\in(0,1)$ be constants, and $\mathscr{A}_{\rho}$ be defined by \eqref{25}. Then for all $\omega\geq\omega_{\rho}$, with $\omega_{\rho}$ defined by 
	\begin{align}\label{2.12}
		\omega_{\rho}=\left(\frac{\alpha L_{\rho}}{\nu}+2^{2\delta-1}\beta(1+\gamma)^2(\delta+1)^2\right),
	\end{align}
	the nonlinear map $\mathscr{B}_{\rho,\omega}:=\mathscr{A}_{\rho}+\omega\I:\X\to\Y^*$ is \emph{monotone} in the sense that 
	\begin{align}
		\langle \mathscr{B}_{\rho,\omega}(v)-\mathscr{B}_{\rho,\omega}(w),v-w\rangle \geq 0\ \text{ for all }\ v,w\in \X. 
	\end{align}
	Moreover, $\mathscr{B}_{\rho,\omega}$ is \emph{hemicontinuous} in the sense that for all $v_1,v_2\in\X$ and $w\in\Y$, \begin{align}\label{2p12}\lim\limits_{\lambda\to 0}\langle\mathscr{B}_{\rho,\omega}(v_1+\lambda v_2),w\rangle=\langle\mathscr{B}_{\rho,\omega}(v_1),w\rangle.\end{align} Finally, $\mathscr{B}_{\rho,\omega}$ is \emph{coercive} in the sense that 
	\begin{align}\label{2p13}\lim\limits_{\|v\|_{\Y\to\infty}}\frac{\langle\mathscr{B}_{\rho,\omega}(v),v\rangle}{\|v\|_{\Y}}=\infty,\end{align}
	for all $v\in\X$. 
\end{lemma}

\begin{proof}
	The proof is divided into the following steps:
	\vskip 0.1cm
	\noindent \textbf{Step 1:} \emph{Monotonicity.}
	Let us first prove the monotoniciy property. For all $v,w\in\X$, since $f_{\rho}$ satisfies a  monotone condition (see \eqref{219} below), $g_1$  is non-decreasing and \eqref{1.5} holds, we have
	\begin{align}\label{28}
		&\langle\mathscr{A}_{\rho}(v)-\mathscr{A}_{\rho}(w),v-w\rangle\nonumber\\&=\nu\|v_x-w_x\|_{\L^2}^2-\frac{\mu}{2}(g_1(v(1))-g_1(w(1)))^2+\frac{\mu}{2}(v_x(0)-w_x(0))^2-\alpha(f_{\rho}(v)-f_{\rho}(w),v_x-w_x)\nonumber\\&\quad-\beta(1+\gamma)(v^{\delta+1}-w^{\delta+1},v-w)+\beta\gamma\|v-w\|_{\L^2}^2+\beta(v^{2\delta+1}-w^{2\delta+1},v-w)\nonumber\\&\quad+\nu (g_1(v(1))-g_1(w(1)))(v(1)-w(1))+\mu (g_2(v(1))-g_2(w(1)))(v(1)-w(1))\nonumber\\&\quad +\alpha(f_{\rho}(v(1))-f_{\rho}(w(1)))(v(1)-w(1))\nonumber\\&\geq \frac{\nu}{2}\|v_x-w_x\|_{\L^2}^2-\frac{\alpha L_{\rho}}{\nu}\|v-w\|_{\L^2}^2+\beta\gamma\|v-w\|_{\L^2}^2\nonumber\\&\quad-\beta(1+\gamma)(v^{\delta+1}-w^{\delta+1},v-w)+\beta(v^{2\delta+1}-w^{2\delta+1},v-w). 
	\end{align}
	
The estimate of the convective term can be justified in the following way:	 Without loss of generality, we may assume that $|u|\leq|v|$ for all $u,v\in\Y$. Then for $|u|,|v|\leq\rho,$ we have 
	\begin{align}
		&	|\langle f_{\rho,x}(u)- f_{\rho,x}(v),u-v\rangle|\nonumber\\&=|(f_{\rho}(u(1))-f_{\rho}(v(1)))(u(1)-v(1))-\langle f_{\rho}(u)-f_{\rho}(v),u_x-v_x\rangle|\nonumber\\&=\left|\frac{1}{\delta+1}\left(u^{\delta+1}(1)-v^{\delta+1}(1)\right)(u(1)-v(1))-\frac{1}{\delta+1}\langle u^{\delta+1}-v^{\delta+1},u_x-v_x\rangle\right|\nonumber\\&\leq\frac{1}{\delta+1}\|u^{\delta+1}-v^{\delta+1}\|_{\L^{\infty}}\|u-v\|_{\L^{\infty}}+\|(\theta u+(1-\theta)v)^{\delta}(u-v)\|_{\L^2}\|u_x-v_x\|_{\L^2}\nonumber\\&\leq \|(\theta u+(1-\theta)v)^{\delta}\|_{\L^{\infty}}\left[\|u-v\|_{\L^{\infty}}^2+\|u-v\|_{\L^2}\|u_x-v_x\|_{\L^2}\right]\nonumber\\&\leq 2^{\delta+1}\left(\|u\|_{\L^{\infty}}^{\delta}+\|v\|_{\L^{\infty}}^{\delta}\right)\|u-v\|_{\L^2}\|u_x-v_x\|_{\L^2}\nonumber\\&\leq\frac{\nu}{2}\|u_x-v_x\|_{\L^2}^2+\frac{2^{4\delta+3}\rho^{2\delta}}{\nu}\|u-v\|_{\L^2}^2,
	\end{align}
	where we have used H\"older's, Agmon's and Young's inequalities.    For $|u|,|v|>\rho,$ we have 
	\begin{align}
		&	|\langle f_{\rho,x}(u)- f_{\rho,x}(v),u-v\rangle|\nonumber\\&=\left|\left[\left(\frac{\rho}{|u(1)|}\right)^{\delta+1}f(u(1))-\left(\frac{\rho}{|v(1)|}\right)^{\delta+1}f(v(1))\right](u(1)-v(1))\right.\nonumber\\&\quad\left.-\left< \left(\frac{\rho}{\|u\|_{\L^{\infty}}}\right)^{\delta+1}f(u)-\left(\frac{\rho}{\|v\|_{\L^{\infty}}}\right)^{\delta+1}f(v),u_x-v_x\right>\right|\nonumber\\&=\Bigg|\left(\frac{\rho}{|u(1)|}\right)^{\delta+1}\frac{1}{\delta+1}(u^{\delta+1}(1)-v^{\delta+1}(1))(u(1)-v(1))\nonumber\\&\quad+\left[\left(\frac{\rho}{|u(1)|}\right)^{\delta+1}-\left(\frac{\rho}{|v(1)|}\right)^{\delta+1}\right]\frac{1}{\delta+1}v^{\delta+1}(1)\nonumber\\&\quad-\left(\frac{\rho}{\|u\|_{\L^{\infty}}}\right)^{\delta+1}\frac{1}{\delta+1}\langle u^{\delta+1}-v^{\delta+1},v_x-u_x\rangle\nonumber\\&\quad-\left[\left(\frac{\rho}{\|u\|_{\L^{\infty}}}\right)^{\delta+1}-\left(\frac{\rho}{\|v\|_{\L^{\infty}}}\right)^{\delta+1}\right]\frac{1}{\delta+1}\langle v^{\delta+1},u_x-v_x\rangle \Bigg|\nonumber\\&=\Bigg|\left(\frac{\rho}{|u(1)|}\right)^{\delta+1}(\theta u(1)+(1-\theta)v(1))^{\delta}(u(1)-v(1))^2\nonumber\\&\quad+\frac{\rho^{\delta+1}}{|u(1)|^{\delta+1}|v(1)|^{\delta+1}}(\theta|v(1)|+(1-\theta)|u(1)|)^{\delta}\left(|v(1)|-|u(1)|\right)(u(1)-v(1))\nonumber\\&\quad-\left(\frac{\rho}{\|u\|_{\L^{\infty}}}\right)^{\delta+1}\langle(\theta u+(1-\theta v))^{\delta}(u-v),u_x-v_x\rangle\nonumber\\&\quad-\frac{\rho^{\delta+1}}{\|u\|_{\L^{\infty}}^{\delta+1}\|v\|_{\L^{\infty}}^{\delta+1}}\left(\theta\|v\|_{\L^{\infty}}+(1-\theta)\|u\|_{\L^{\infty}}\right)^{\delta}\left(\|v\|_{\L^{\infty}}-\|u\|_{\L^{\infty}}\right)\langle v^{\delta+1},u_x-v_x\rangle\Bigg|\nonumber\\&\leq 2^{2\delta}\rho^{\delta}\left[\|u-v\|_{\L^{\infty}}^2+\|u-v\|_{\L^2}\|u_x-v_x\|_{\L^2}+\|u-v\|_{\L^{\infty}}\|u_x-v_x\|_{\L^2}\right]\nonumber\\&\leq \frac{\nu}{2}\|u_x-v_x\|_{\L^2}^2+\frac{L_{\rho}}{\nu}\|u-v\|_{\L^2}^2,
	\end{align}
	where we have used H\"older's, Agmon's, Sobolev's and Young's inequalities. 
	Finally, for $|u|\leq\rho$ and $|v|>\rho$, we deduce 
	\begin{align}
		&	|\langle f_{\rho,x}(u)- f_{\rho,x}(v),u-v\rangle|\nonumber\\&=\left|\left[f(u(1))-\left(\frac{\rho}{|v(1)|}\right)^{\delta+1}f(v(1))\right](u(1)-v(1))-\left< f(u)-\left(\frac{\rho}{\|v\|_{\L^{\infty}}}\right)^{\delta+1}f(v),u_x-v_x\right>\right|\nonumber\\&=\Bigg|\left[1-\left(\frac{\rho}{|v(1)|}\right)^{\delta+1}\right]f(u(1))(u(1)-v(1))+\left(\frac{\rho}{|v(1)|}\right)^{\delta+1}(f(u(1))-f(v(1)))(u(1)-v(1))\nonumber\\&\quad-\left[1-\left(\frac{\rho}{\|v\|_{\L^{\infty}}}\right)^{\delta+1}\right]\langle f(u),u_x-v_x\rangle-\left(\frac{\rho}{\|v\|_{\L^{\infty}}}\right)^{\delta+1}\left<(f(u)-f(v)),u_x-v_x\right>\Bigg|\nonumber\\&\leq \left[\left(\frac{\rho}{|u(1)|}\right)^{\delta+1}-\left(\frac{\rho}{|v(1)|}\right)^{\delta+1}\right]|f(u(1))||u(1)-v(1)|\nonumber\\&\quad+\left(\frac{\rho}{|v(1)|}\right)^{\delta+1}|(f(u(1))-f(v(1)))||u(1)-v(1)|\nonumber\\&\quad+\left[\left(\frac{\rho}{\|u\|_{\L^{\infty}}}\right)^{\delta+1}-\left(\frac{\rho}{\|v\|_{\L^{\infty}}}\right)^{\delta+1}\right]\|f(u)\|_{\L^2}\|u_x-v_x\|_{\L^2}\nonumber\\&\quad+\left(\frac{\rho}{\|v\|_{\L^{\infty}}}\right)^{\delta+1} \|f(u)-f(v)\|_{\L^2}\|u_x-v_x\|_{\L^2}\nonumber\\&\leq \left(\frac{\rho}{|v(1)|}\right)^{\delta+1}(\theta|v(1)|+(1-\theta)|u(1)|)^{\delta}(|v(1)|-|u(1)|)|u(1)-v(1)|\nonumber\\&\quad+\left(\frac{\rho}{|v(1)|}\right)^{\delta+1}|\theta u(1)+(1-\theta)v(1)|^{\delta}|u(1)-v(1)|^2\nonumber\\&\quad+\left(\frac{\rho}{\|v\|_{\L^{\infty}}}\right)^{\delta+1}(\theta\|v\|_{\L^{\infty}}+(1-\theta)\|u\|_{\L^{\infty}})^{\delta}(\|v\|_{\L^{\infty}}-\|u\|_{\L^{\infty}})\|u_x-v_x\|_{\L^2}\nonumber\\&\quad+\left(\frac{\rho}{\|v\|_{\L^{\infty}}}\right)^{\delta+1}\|\theta u+(1-\theta)v\|_{\L^{\infty}}^{\delta}\|u-v\|_{\L^{2}}\|u_x-v_x\|_{\L^2}\nonumber\\&\leq 2^{2\delta}\rho^{\delta}\left[\|u-v\|_{\L^{\infty}}^2+\|u-v\|_{\L^{\infty}}\|u_x-v_x\|_{\L^2}+\|u-v\|_{\L^{2}}\|u_x-v_x\|_{\L^2}\right]
		\nonumber\\&\leq \frac{\nu}{2}\|u_x-v_x\|_{\L^2}^2+\frac{L_{\rho}}{\nu}\|u-v\|_{\L^2}^2.
	\end{align}
	
	From the above calculations, one can conclude that 
	\begin{align}\label{219}
		&	|\langle f_{\rho,x}(u)- f_{\rho,x}(v),u-v\rangle|\leq \frac{\nu}{2}\|u_x-v_x\|_{\L^2}^2+\frac{L_{\rho}}{\nu}\|u-v\|_{\L^2}^2,
	\end{align}
	for all $u,v\in\Y$. 
	
Let us now estimate the term $-\beta(v^{2\delta+1}-w^{2\delta+1},v-w)$ from \eqref{28} as
	\begin{align}\label{2.9}
		&\beta(v^{2\delta+1}-w^{2\delta+1},v-w)\nonumber\\&= \beta(v^{2\delta},(v-w)^2) +\beta(w^{2\delta},(v-w)^2)+\beta(v^{2\delta}w-w^{2\delta}v,v-w)\nonumber\\&=\beta\|v^{\delta}(v-w)\|_{\L^2}^2+\beta\|w^{\delta}(v-w)\|_{\L^2}^2+\beta(vw,v^{2\delta}+w^{2\delta})-\beta(v^2,w^{2\delta})-\beta(w^2,v^{2\delta})\nonumber\\&=\frac{\beta}{2}\|v^{\delta}(v-w)\|_{\L^2}^2+\frac{\beta}{2}\|w^{\delta}(v-w)\|_{\L^2}^2+\frac{\beta}{2}(v^{2\delta}-w^{2\delta},v^2-w^2)\nonumber\\&\geq \frac{\beta}{2}\|v^{\delta}(v-w)\|_{\L^2}^2+\frac{\beta}{2}\|w^{\delta}(v-w)\|_{\L^2}^2,
	\end{align}
	since $(v^{2\delta}-w^{2\delta},v^2-w^2)\geq 0$. Using Taylor's formula, H\"older's and Young's inequalities, we estimate the term $\beta(1+\gamma)(v^{\delta+1}-w^{\delta+1},v-w)$ from \eqref{28} as 
	\begin{align}\label{2.10}
		&\beta(1+\gamma)(v^{\delta+1}-w^{\delta+1},v-w)\nonumber\\&=\beta(1+\gamma)(\delta+1)((\theta v+(1-\theta)w)^{\delta}(v-w),v-w)\nonumber\\&\leq \beta(1+\gamma)(\delta+1)2^{\delta-1}(\|v^{\delta}(v-w)\|_{\L^2}+\|w^{\delta}(v-w)\|_{\L^2})\|v-w\|_{\L^2}\nonumber\\&\leq\frac{\beta}{4}\|v^{\delta}(v-w)\|_{\L^2}^2+\frac{\beta}{4}\|w^{\delta}(v-w)\|_{\L^2}^2+\frac{\beta}{2}2^{2\delta}(1+\gamma)^2(\delta+1)^2\|v-w\|_{\L^2}^2.
	\end{align}
	It should also be noted that  for all $v,w\in\L^{2(\delta+1)}(0,1)$, 
	\begin{align}\label{227}
		\|v-w\|_{\L^{2(\delta+1)}}^{2(\delta+1)}&=\int_0^1|v(x)-w(x)|^{2\delta}|v(x)-w(x)|^2dx\nonumber\\&\leq 2^{2\delta-1}\left[\|v^{\delta}(v-w)\|_{\L^2}^2+\|w^{\delta}(v-w)\|_{\L^2}^2\right].
	\end{align}
	Combining \eqref{2.9}-\eqref{227} and substituting it in \eqref{28},  we obtain 
	\begin{align}\label{2.11}
		&\langle\mathscr{A}_{\rho}(v)-\mathscr{A}_{\rho}(w),v-w\rangle\nonumber\\&\geq \frac{\nu}{2}\|v_x-w_x\|_{\L^2}^2-\frac{\alpha L_{\rho}}{\nu}\|v-w\|_{\L^2}^2+\beta\gamma\|v-w\|_{\L^2}^2\nonumber\\&\quad+\frac{\beta}{4}\|v^{\delta}(v-w)\|_{\L^2}^2+\frac{\beta}{4}\|w^{\delta}(v-w)\|_{\L^2}^2-2^{2\delta-1}\beta(1+\gamma)^2(\delta+1)^2\|v-w\|_{\L^2}^2\nonumber\\&\geq \max\left\{\frac{\nu}{2},\beta\gamma\right\}\|v-w\|_{\H^1}^2+\frac{\beta}{2^{2\delta+1}}\|v-w\|_{\L^{2(\delta+1)}}^{2(\delta+1)}-\omega_{\rho}\|v-w\|_{\L^2}^2,
	\end{align}
	where $\omega_{\rho}$ is defined in \eqref{2.12}. Therefore for all $\omega\geq\omega_{\rho}$, we deduce from \eqref{2.11} that 
	\begin{align}\label{2.13}
		\langle\mathscr{A}_{\rho}(v)-\mathscr{A}_{\rho}(w),v-w\rangle+\omega\|v-w\|_{\L^2}^2\geq \max\left\{\frac{\nu}{2},\beta\gamma\right\}\|u-v\|_{\H^1}^2,
	\end{align}
	for all $v,w\in\X$, so that the monotonicity of the operator $\mathscr{B}_{\rho,\omega}=\mathscr{A}_{\rho}+\omega\I$ follows.

	\vskip 0.1cm
	\noindent \textbf{Step 2:} \emph{Hemicontinuity.} Note that for  reflexive Banach spaces, the demicontinuity property  implies  hemicontinuity \cite[Section 3.3]{GDJM}. As $\X$ is a reflexive Banach space (Lemma \ref{lem2p1}), it is enough to show that $\mathscr{B}_{\rho,\omega}$ is demicontinuous  in the sense that $v_n\to v$ in $\X$ implies $\mathscr{B}_{\rho,\omega}(v_n)\xrightharpoonup{w}\mathscr{B}_{\rho,\omega}(v)$ in $\Y^*$ as $n\to\infty$. 
	\vskip 0.05 cm
	\noindent
	\textbf{Claim:} \emph{The operator $\mathscr{B}_{\rho,\omega}$ is demicontinuous.}
	 In order to show  the demicontinuity of the operator $\mathscr{B}_{\rho,\omega}$, it is enough to prove the demicontinuity of the operator $\mathscr{A}_{\rho}(\cdot)$ defined in \eqref{25}, that is, we need to show that if $v_n\to v$ in $\X$ implies $\mathscr{A}_{\rho}(v_n)\xrightharpoonup{w}\mathscr{A}_{\rho}(v)$ in $\Y^*$ as $n\to\infty$. We first consider the case $|v_n|,|v|\leq\rho$.   We choose a sequence $\{v_n\}_{n\in\N}\in\X$ such that $v_n\to v$ in $\X$ and $|v_n|,|v|\leq\rho$. Since $\H^1(0,1)\hookrightarrow\C([0,1])$, the convergence $v_n\to v$ in $\X$ implies $v_n(x)\to v(x)$ for all $x\in[0,1]$.   For any $w\in \Y$, we consider 
	\begin{align}\label{212}
	&	\left|	\langle \mathscr{A}(v_n)-\mathscr{A}(v),w\rangle \right|\nonumber\\&\leq\nu\big| \big((v_{n,x}-v_x),w_x\big)\big|+\nu |(g_1(v_n(1))-g_1(v(1)))w(1)|\nonumber\\&\quad+\mu|((v_{n,xx}-v_{xx}),w_x)|+\mu |(g_2(v_n(1))-g_2(v(1)))w(1)|\nonumber\\&\quad+\alpha|(f(v_n)-f(v),w_x)|+\alpha|(f(v_n(1))-f(v(1)))w(1)|\nonumber\\&\quad+\beta(1+\gamma)|(v_n^{\delta+1}-v^{\delta+1},w)|+\beta|(v_n^{2\delta+1}-v^{2\delta+1},w)|+\beta\gamma|(v_n-v,w)|\nonumber\\&\leq \nu\|v_{n,x}-v_x\|_{\L^2}\|w_x\|_{\L^2}+\eta|v_n(1)-v(1)||w(1)|\nonumber\\&\quad+\frac{\alpha(\delta+1)}{(\delta+2)}|\theta v_n(1)+(1-\theta)v(1)|^{\delta}|v_n(1)-v(1)||w(1)|+\mu\|v_{n,xx}-v_{xx}\|_{\L^2}\|w_x\|_{\L^2}\nonumber\\&\quad+\frac{\delta\mu}{\nu^2}\bigg[\eta^2|v_n(1)-v(1)|+\frac{2\alpha\eta(\delta+1)}{(\delta+2)}\theta v_n(1)+(1-\theta)v(1)|^{\delta}|v_n(1)-v(1)|\nonumber\\&\quad+\frac{\alpha^2(2\delta+1)}{(\delta+2)^2}|\theta v_n(1)+(1-\theta)v(1)|^{2\delta}|v_n(1)-v(1)|\bigg]|w(1)|\nonumber\\&\quad+\alpha\|\theta v_n+(1-\theta)v\|_{\L^{\infty}}^{\delta}\|v_n-v\|_{\L^2}\|w_x\|_{\L^2}+\alpha|\theta v_n(1)+(1-\theta)v(1)|^{\delta}|v_n(1)-v(1)||w(1)|\nonumber\\&\quad+\beta\Big[(1+\gamma)(\delta+1)\|\theta v_n+(1-\theta)v\|_{\L^{\infty}}^{\delta}+(2\delta+1)\|\theta v_n+(1-\theta)v\|_{\L^{\infty}}^{2\delta}+1\Big]\|v_n-v\|_{\L^2}\|w\|_{\L^2}
	\nonumber\\&\to 0\ \text{ as }\ n\to\infty,
	\end{align}
	for $\delta=2$. Similarly, one can show for the case $\delta=1$ also. For the other cases, the proof follows in a similar way except for the convective term with cutoff function. For $|v_n|,|v|>\rho$, we have 
	\begin{align}
	&	|(f_{\rho}(v_n)-f_{\rho}(v),w_x)|+|(f_{\rho}(v_n(1))-f_{\rho}(v(1)))w(1)|\nonumber\\&=\left|\left(\left(\frac{\rho}{\|v_n\|_{\L^{\infty}}}\right)^{\delta+1}f(v_n)-\left(\frac{\rho}{\|v\|_{\L^{\infty}}}\right)^{\delta+1}f(v),w_x\right)\right|\nonumber\\&\quad+\left|\left[\left(\frac{\rho}{|v_n(1)|}\right)^{\delta+1}f(v_n(1))-\left(\frac{\rho}{|v(1)|}\right)^{\delta+1}f(v(1))\right]w(1)\right|
	\nonumber\\&=\left|\left(\left[\left(\frac{\rho}{\|v_n\|_{\L^{\infty}}}\right)^{\delta+1}-\left(\frac{\rho}{\|v\|_{\L^{\infty}}}\right)^{\delta+1}\right]f(v_n)+\left(\frac{\rho}{\|v\|_{\L^{\infty}}}\right)^{\delta+1}(f(v_n)-f(v)),w_x\right)\right|\nonumber\\&\quad+\left|\left\{\left[\left(\frac{\rho}{|v_n(1)|}\right)^{\delta+1}-\left(\frac{\rho}{|v(1)|}\right)^{\delta+1}\right]f(v_n(1))+\left(\frac{\rho}{|v(1)|}\right)^{\delta+1}(f(v_n(1))-f(v(1)))\right\}w(1)\right|\nonumber\\&\leq \left(\frac{\rho}{\|v\|_{\L^{\infty}}}\right)^{\delta+1}\Big[ (\theta\|v\|_{\L^{\infty}}+(1-\theta)\|v_n\|_{\L^{\infty}})^{\delta}\|v_n-v\|_{\L^{\infty}}\nonumber\\&\quad+\|(\theta v_n+(1-\theta)v)^{\delta}\|_{\L^{\infty}}\|v_n-v\|_{\L^2}\Big]\|w_x\|_{\L^2}\nonumber\\&\quad+2\left(\frac{\rho}{|v(1)|}\right)^{\delta+1}\Big[(\theta|v_n(1)|+(1-\theta)|v(1))^{\delta}|v_n(1)-v(1)|\Big]|w(1)|\nonumber\\&\leq 2^{\delta+1}(\|v_n\|_{\L^{\infty}}^{\delta}+\|v\|_{\L^{\infty}}^{\delta})\|v_n-v\|_{\L^{\infty}}\|w\|_{\Y}\nonumber\\&\to 0\ \text{ as }\ n\to\infty. 
	\end{align}
	For the case $|v_n|\leq \rho$ and $|v|>\rho$, we infer 
	\begin{align}
		&	|(f_{\rho}(v_n)-f_{\rho}(v),w_x)|+|(f_{\rho}(v_n(1))-f_{\rho}(v(1)))w(1)|\nonumber\\&=\left|\left(f(v_n)-\left(\frac{\rho}{\|v\|_{\L^{\infty}}}\right)^{\delta+1}f(v),w_x\right)\right|+\left|\left[f(v_n(1))-\left(\frac{\rho}{|v(1)|}\right)^{\delta+1}f(v(1))\right]w(1)\right|\nonumber\\&= \left|\left(\frac{\|v\|_{\L^{\infty}}^{\delta+1}-\rho^{\delta+1}}{\|v\|_{\L^{\infty}}^{\delta+1}}\right)(f(v_n),w_x)+\left(\frac{\rho}{\|v\|_{\L^{\infty}}}\right)^{\delta+1}(f(v_n)-f(v),w_x)\right|\nonumber\\&\quad+\left|\left(\frac{|v(1)|^{\delta+1}-\rho^{\delta+1}}{|v(1)|^{\delta+1}}\right)f(v_n(1))w(1)+\left(\frac{\rho}{|v(1)|}\right)^{\delta+1}(f(v_n(1))-f(v(1)))w(1)\right|\nonumber\\&\leq\left[\left(\theta\|v\|_{\L^{\infty}}+(1-\theta)\|v_n\|_{\L^{\infty}}\right)^{\delta}\|v_n-v\|_{\L^{\infty}}+\|\theta v_n+(1-\theta)v\|_{\L^{\infty}}^{\delta}\|v_n-v\|_{\L^2}\right]\|w_x\|_{\L^2}\nonumber\\&\quad+\left[(\theta|v(1)|^{\delta}+(1-\theta)|v_n|^{\delta})+|\theta v_n(1)+(1-\theta)v(1)|^{\delta}\right]|v_n(1)-v(1)||w(1)|\nonumber\\&\to 0\ \text{ as }\ n\to\infty. 
	\end{align}
A similar calculation holds for the case, $|v_n|> \rho$ and $|v|\leq\rho$ also. 
Thus the operator $\mathscr{A}(\cdot)$ is demicontinuous and hence hemicontinuous also.
	
	\vskip 0.1cm
	\noindent \textbf{Step 3:} \emph{Coercivity.} By taking $w=0$ in \eqref{2.11}, one can easily see that 
	\begin{align}\label{2p16}
		\frac{\langle\mathscr{B}_{\rho,\omega}(v),v\rangle}{\|v\|_{\H^1}}\geq \frac{1}{2}\max\left\{\frac{\nu}{2},\beta\gamma\right\}\|v\|_{\H^1}-\|\mathscr{A}_{\rho}(0)\|_{\Y^*}\to \infty
	\end{align}
	for all $v\in\X$ such that $\|v\|_{\H^1}\to\infty$, and the coercivity of $\mathscr{B}_{\rho,\omega}(\cdot)$  follows. 
\end{proof}

Finally, we have the following result on the existence and uniqueness of solutions  for the problem \eqref{1.1}-\eqref{1.2} with the controls \eqref{1p3} and \eqref{1.3} for $\delta=1$ and $\delta=2$, respectively. 
\begin{theorem}\label{thm2.5}
	For any initial data $u_0\in\D(\mathscr{A})$, the system \eqref{2.3} possesses a \emph{unique strong solution} $u\in\C^1([0,\infty);\L^2(0,1))\cap\C([0,\infty);\H^3(0,1))$. 
\end{theorem}
\begin{proof}
	The proof is divided into the following steps:
	\vskip 0.1 cm
	\noindent\textbf{Step 1:} \emph{Strong solution of the approximate problem:} 
	We use the well-known Crandall-Liggett    and Minty Theorems to establish this result. Since $\mathscr{B}_{\omega,\rho}$ is monotone, hemicontinuous and coercive from $\X$ to $\Y^*$, then by an application of Theorem \ref{thm2.2} (Minty's Theorem) yields that $ \mathscr{B}_{\omega,\rho}$ is onto, that is,  $\mathrm{Range}(\mathscr{B}_{\omega,\rho})=\Y^*$. Since  $\mathscr{B}_{\rho,\omega}:\X\to\Y^*$ is monotone, hemicontinuous and coercive,  a further application of \cite[Example 2.3.7]{HB} implies  that $\mathscr{B}_{\rho,\omega}$ (the restriction of $\mathscr{B}_{\rho,\omega}$ to $\H$) is maximal monotone in $\H$ with $\D(\mathscr{A}_{\rho})=\D(\mathscr{B}_{\omega,\rho})$. Since $\H^*=\H$, from \cite[Theorem 3.1]{VB1}, we infer that the maximal monotone and $m$-accretive sets coincide. Therefore,   for sufficiently large $\omega$,  $\mathscr{B}_{\omega,\rho}$ is $m$-accretive with domain $\D(\mathscr{A}_{\rho})=\D(\mathscr{B}_{\omega,\rho})$.  The above arguments can be justified in the following way: Since $\mathrm{Range}(\mathscr{B}_{\omega,\rho})=\Y^*$, for each $\xi\in\Y^*$, there exists a unique $\varphi\in\X$ such that $\mathscr{B}_{\omega,\rho}(\varphi)=\xi$ in the distributional sense, which means 
	\begin{align}
		\langle\mathscr{B}_{\omega,\rho}(\varphi),v\rangle=\langle\xi,v\rangle,\ \text{ for all }\ v\in\Y. 
	\end{align}
	Furthermore, if $\xi\in\L^2(0,1),$ then $\varphi\in\H^3(0,1)$ and it satisfies (cf. \cite{BGGL})
	\begin{equation}\label{2.21}
		\left\{
		\begin{aligned}
			\omega\varphi(x)-\nu\varphi_{xx}(x)+\mu\varphi_{xxx}(x)+\alpha(f_{\rho}(\varphi(x)))_x-\beta \varphi(x)(1-\varphi^{\delta}(x))(\varphi^{\delta}(x)-\gamma)&=\xi(x),\\
			\varphi(0)=0,\  \varphi_x(1)+g_1(\varphi(1))=0,\  \varphi_{xx}(1)-g_2(\varphi(1))&=0,
		\end{aligned}
		\right. 
	\end{equation}
	where the first equation in \eqref{2.21} holds for a.e. $x\in [0,1]$. In fact, taking the inner product with $\varphi$ to the first equation in \eqref{2.21}, we find 
	\begin{align}\label{240}
		\omega\|\varphi\|_{\L^2}^2+\langle\mathscr{A}_{\rho}(\varphi),\varphi\rangle&=\langle\xi,\varphi\rangle.
	\end{align}
	Applying \eqref{215} and \eqref{2k12} in \eqref{240}, we deduce 
	\begin{align}\label{2p41}
		\left(\omega-\frac{\beta}{4\theta}(1-2(2\theta-1)\gamma+\gamma^2)^2\right)\|\varphi\|_{\L^2}^2+\frac{\nu}{2}\|\varphi_x\|_{\L^2}^2+\beta\theta\|\varphi\|_{\L^{2(\delta+1)}}^{2(\delta+1)}\leq\frac{1}{2\nu}\|\xi\|_{\Y^*}^2,
	\end{align}
for some $0<\theta< 1$ and sufficiently large $\omega$. An application of Agmon's inequality (Lemma \ref{lem21}) yields 
\begin{align}\label{2a42}
	\max_{x\in[0,1]}|\varphi(x)|\leq\sqrt{2}\|\varphi\|_{\L^2}^{1/2}\|\varphi_x\|_{\L^2}^{1/2}\leq \frac{1}{\sqrt{\nu}}\|\xi\|_{\Y^*}. 
\end{align}
	Taking the inner product with $-\varphi_{xx}$ to the first equation in \eqref{2.21}, we obtain 
	\begin{align}\label{2p43}
		\omega\|\varphi_x\|_{\L^2}^2+(\mathscr{A}_{\rho}(\varphi),-\varphi_{xx})=(\xi,-\varphi_{xx}).
	\end{align}
Since $\frac{\rho}{\|\varphi\|_{\L^{\infty}}}<1$ for $|\varphi|>\rho$, we need to consider the case  $|\varphi|\leq\rho$ only as the calculations for the other case be performed in a similar way. 	For $|\varphi|\leq\rho$, we have 
	\begin{align}\label{230}
		&(\mathscr{A}(\varphi),-\varphi_{xx})\nonumber\\&=\nu\|\varphi_{xx}\|_{\L^2}^2-\frac{\mu}{2}g_2^2(\varphi(1))+\frac{\mu}{2}\varphi_{xx}^2(0)-\alpha(\varphi^{\delta}\varphi_{x},\varphi_{xx})+\beta(1+\gamma)(\varphi^{\delta+1},\varphi_{xx})\nonumber\\&\quad+\beta\gamma\|\varphi_{x}\|_{\L^2}^2+\beta\gamma g_1(\varphi(1))\varphi(1)+\beta(2\delta+1)\|\varphi^{\delta}\varphi_{x}\|_{\L^2}^2+\beta g_1(\varphi(1))\varphi^{2\delta+1}(1).
	\end{align}
	We estimate $\alpha|(\varphi^{\delta}\varphi_{x},\varphi_{xx})|$ as 
	\begin{align}\label{231}
		\alpha|(\varphi^{\delta}\varphi_{x},\varphi_{xx})|\leq\alpha\|\varphi_{xx}\|_{\L^2}\|\varphi^{\delta}\|_{\L^{\infty}}\|\varphi_{x}\|_{\L^2}\leq\frac{\nu}{4}\|\varphi_{xx}\|_{\L^2}^2+\frac{\alpha^2}{\nu}\|\varphi\|_{\L^{\infty}}^{2\delta}\|\varphi_{x}\|_{\L^2}^2.
	\end{align}
	Similarly, we estimate $\beta(1+\gamma)|(\varphi^{\delta+1},\varphi_{xx})|$ as 
	\begin{align}\label{232}
		\beta(1+\gamma)|(\varphi^{\delta+1},\varphi_{xx})|\leq\beta(1+\gamma)\|\varphi^{\delta+1}\|_{\L^2}\|\varphi_{xx}\|_{\L^2}\leq\frac{\nu}{4}\|\varphi_{xx}\|_{\L^2}^2+\frac{\beta^2(1+\gamma)^2}{\nu}\|\varphi\|_{\L^{2(\delta+1)}}^{2(\delta+1)}.
	\end{align}
	Substituting  \eqref{230}-\eqref{232} in \eqref{2p43} leads to 
	\begin{align}\label{243}
	&\omega\|\varphi_x\|_{\L^2}^2+	\frac{\nu}{4}\|\varphi_{xx}\|_{\L^2}^2+\beta(2\delta+1)\|\varphi^{\delta}\varphi\|_{\L^2}^2\nonumber\\&\leq\frac{1}{\nu}\|\xi\|_{\L^2}^2+\frac{\mu}{2}g_2^2(\varphi)+\frac{\alpha^2}{\nu}\|\varphi\|_{\L^{\infty}}^{2\delta}\|\varphi_x\|_{\L^2}^2+\frac{\beta^2(1+\gamma)^2}{\nu}\|\varphi\|_{\L^{2(\delta+1)}}^{2(\delta+1)}.
	\end{align}
From the definition of controls given in \eqref{1.3} and \eqref{1p3}, 	using Agmon's inequality (Lemma \ref{lem21}) and the estimate \eqref{2a42}, we deduce from \eqref{243} that 
\begin{align}\label{2p44}
	\omega\|\varphi_x\|_{\L^2}^2+	\frac{\nu}{4}\|\varphi_{xx}\|_{\L^2}^2+\beta(2\delta+1)\|\varphi^{\delta}\varphi\|_{\L^2}^2&\leq C_{\nu,\mu,\alpha,\beta,\gamma,\delta}\left(1+\|\xi\|_{\L^2}^{8\delta}\right),
\end{align}
where we have used \eqref{2p41}. Taking the inner product with $\varphi_{xxx}$ to the first equation in \eqref{2.21}, we get 
\begin{align}
	(\mathscr{A}_{\rho}(\varphi),\varphi_{xxx})=(\xi,\varphi_{xxx})-\omega(\varphi,\varphi_{xxx}).
\end{align}
An integration by parts and  application of H\"older's and Young's inequalities yield
\begin{align}\label{246}
	&\mu\|\varphi_{xxx}\|_{\L^2}^2+\frac{\nu}{2}\varphi_{xx}^2(0)\nonumber\\&=(\xi,\varphi_{xxx})-(\omega+\beta\gamma)(\varphi,\varphi_{xxx})-\alpha(\varphi^{\delta}\varphi_x,\varphi_{xxx})+\beta(1+\gamma)(\varphi^{\delta+1},\varphi_{xxx})-\beta(\varphi^{2\delta+1},\varphi_{xxx})\nonumber\\&\leq\frac{3\mu}{4}\|\varphi_{xxx}\|_{\L^2}^2+\frac{1}{\mu}\|\xi\|_{\L^2}^2+\frac{2(\omega+\beta\gamma)^2}{\mu}\|\varphi\|_{\L^2}^2+\frac{2\alpha^2}{\mu}\|\varphi^{\delta}\varphi\|_{\L^2}^2+\frac{2\beta^2(1+\gamma)^2}{\mu}\|\varphi\|_{\L^{2(\delta+1)}}^{2(\delta+1)}\nonumber\\&\quad+\frac{2\beta^2}{\mu}\|\varphi\|_{\L^{4\delta+2}}^{4\delta+2}.
	\end{align}
	Using \eqref{2p44} in \eqref{246}, we finally arrive at 
	\begin{align}\label{2p47}
		&\frac{\mu}{4}\|\varphi_{xxx}\|_{\L^2}^2\leq C_{\nu,\mu,\alpha,\beta,\gamma,\delta,\omega}\left(1+\|\xi\|_{\L^2}^{8\delta}\right).
	\end{align}
	Combining \eqref{2p41}, \eqref{2p44} and \eqref{2p47}, one can easily seen that $$\|\mathscr{B}_{\omega,\rho}(\varphi)\|_{\L^2}\leq C_{\nu,\mu,\alpha,\beta,\gamma,\delta,\omega}\left(1+\|\xi\|_{\L^2}^{4\delta}\right),$$ that is, $\H\subseteq\mathrm{Range}(\mathscr{B}_{\omega,\rho}(\varphi))$. Taking the inner product with $\xi$ to the first equation in \eqref{2.21}, we infer 
	\begin{align*}
		\|\xi\|_{\L^2}\leq\|\mathscr{B}_{\omega,\rho}(\varphi)\|_{\L^2},
	\end{align*}
	so that $\mathrm{Range}(\mathscr{B}_{\omega,\rho}(\varphi))\subseteq\H$. Therefore, it is immediate that $\mathrm{Range}(\mathscr{B}_{\omega,\rho}(\varphi))=\H$. Thus, for each $\lambda\in(0,\frac{1}{\omega_{\rho}})$, there exists a $\varphi\in\X$ such that $\mathscr{B}_{\rho,\frac{1}{\lambda}}(\varphi)=\frac{\xi}{\lambda}$, which yields  $(\I+\lambda\mathscr{A}_{\rho})(\varphi)=\xi$, so that $\mathrm{Range}(\I+\lambda\mathscr{A}_{\rho})=\H$.

%Note that $\mathscr{B}_{\rho,\omega}(v)=-\nu v_{xx}+\mu v_{xxx}+\alpha f_{\rho,x}-\beta v(1-v^{\delta})(v^{\delta}-\gamma)+\omega\I:\X\to\Y^*$. 

Then by an application of the Crandall-Liggett  Theorem (\cite[Theorem 4.1.3, 4.1.4]{VB}, \cite{MGTL}, \cite[Theorems 5.1, 5.2]{JAW}), we have that the problem 
\begin{equation}\label{225}
	\left\{
	\begin{aligned}
		\frac{du_{\rho}(t)}{dt}+\mathscr{A}_{\rho}(u_{\rho})(t)&=0,\ t>0,\\
		u_{\rho}(0)&=u_0,
	\end{aligned}\right.
\end{equation}
has a \emph{unique strong solution} $u_{\rho}\in\C\left([0,\infty);\D(\mathscr{A}_{\rho})\right)\cap\C^1\left((0,\infty);\L^2(0,1)\right)\subset \C\left([0,\infty);\H^3(0,1)\right)\cap\C^1\left((0,\infty);\L^2(0,1)\right) $ for all $\rho>0$ and $u_0\in\D(\mathscr{A})$. Moreover, \eqref{225} has a unique mild solution
$$u_{\rho}(t)=\mathrm{S}_{\rho}(t)(u_0)=\lim\limits_{n\to\infty}\left(\I+\frac{t}{n}\mathscr{A}_{\rho}\right)^{-n}(u_0),$$
 which is a limit in $\C([0,T];\L^2(0,1))$ of a sequence of strong solutions (\cite[Theorem 4.1.3]{VB}).

\vskip 0.1 cm
\noindent\textbf{Step 2:} \emph{Uniform bounds for the solutions:} 
Let us now show the existence and uniqueness of solutions  for the problem \eqref{2.3}. We first establish  the uniform boundedness of the sequence of solutions $\{u_{\rho}\}_{\rho>0}$. Taking the inner product with $u_{\rho}$ in \eqref{225} and using a calculation similar to \eqref{2k12}, one obtains 
\begin{align}\label{239}
	\|u_{\rho}(t)\|_{\L^2}^2&=\|u_0\|_{\L^2}^2-2\int_0^t\langle\mathscr{A}_{\rho}(u_{\rho})(s),u_{\rho}(s)\rangle d s\nonumber\\&\leq \|u_0\|_{\L^2}^2-2\nu\int_0^t\|u_{\rho,x}(s)\|_{\L^2}^2ds -\eta\int_0^t u_{\rho}^2(1,s)d s-\frac{\alpha^2}{\eta(\delta+2)^2}\int_0^tu_{\rho}^{2(\delta+1)}(1,s)ds\nonumber\\&\quad-\mu\int_0^t(g_1(u_{\rho}(1,s)))^2ds-\beta\int_0^t\|u_{\rho}(s)\|_{\L^{2(\delta+1)}}^{2(\delta+1)}ds+\beta(1+\gamma^2)\int_0^t\|u_{\rho}(s)\|_{\L^2}^2ds,
\end{align}
for all $t\in[0,T]$ and $|u_{\rho}|\leq\rho$. A similar expression involving  $\left(\frac{\rho}{\|u_{\rho}\|_{\L^{\infty}}}\right)^{\delta+1}$ holds for $|u_{\rho}|>\rho$ also. Since $\frac{\rho}{\|u_{\rho}\|_{\L^{\infty}}}<1$ for $|u_{\rho}|>\rho$, we need to consider the case  $|u_{\rho}|\leq\rho$ only. An application of Gronwall's inequality in \eqref{239} yields 
\begin{align}
	&	\|u_{\rho}(t)\|_{\L^2}^2+2\nu\int_0^t\|u_{\rho,x}(s)\|_{\L^2}^2ds+\frac{\alpha^2}{\eta(\delta+2)^2}\int_0^tu_{\rho}^{2(\delta+1)}(1,s)ds\nonumber\\&\quad+\frac{\mu}{\delta}\int_0^tg_2(u_{\rho}(1,s))u_{\rho}(1,s)ds+\beta\int_0^t\|u_{\rho}(s)\|_{\L^{2(\delta+1)}}^{2(\delta+1)}ds\leq \|u_0\|_{\L^2}^2e^{\beta T(1+\gamma^2)},
\end{align}
for all $t\in[0,T]$, and the right hand side is independent of $\rho$. 

%Taking the inner product with $|u_{\rho}|^{2p-2}u_{\rho}$ in \eqref{225}, we obtain for all $t\in[0,T]$
%\begin{align}
%	\|u_{\rho}(t)\|_{\L^{2p}}^{2p}&=\|u_{0}\|_{\L^{2p}}^{2p}-2p\int_0^t\langle\mathscr{A}_{\rho}u_{\rho}(s),|u_{\rho}(s)|^{2p-2}u_{\rho}(s)\rangle ds.
%\end{align}

Taking the inner product with $-u_{\rho,xx}$ in \eqref{225}, we find  for a.e.  $t\in[0,T]$
\begin{align}\label{241}
	-	\left(u_{\rho,t}(t),u_{\rho,xx}(t)\right)=(\mathscr{A}_{\rho}u_{\rho}(t),u_{\rho,xx}(t)).
\end{align}
We first consider the case $\delta=1$. Note that 
\begin{align}\label{229}
	-&	\int_0^1u_{\rho,t}(t,x)u_{\rho,xx}(t,x)dx\nonumber\\&=-u_{\rho,t}(t,x)u_{\rho,x}(t,x)\big|_{0}^1+\int_0^1u_{\rho,xt}(t,x)u_{\rho,x}(t,x)dx\nonumber\\&=\frac{1}{2}\frac{d}{dt}\|u_{\rho,x}(t)\|_{\L^2}^2+\frac{1}{\nu} u_{\rho,t}(1,t)\left(\eta+\frac{\alpha^2}{\eta(\delta+2)^2}u_{\rho}^{2\delta}(1,t)\right)u_{\rho}(1,t)\nonumber\\&=\frac{1}{2}\frac{d}{dt}\left[\|u_{\rho,x}(t)\|_{\L^2}^2+\frac{\eta}{\nu}u_{\rho}^2(1,t)+\frac{\alpha^2}{\nu\eta(\delta+2)^2(\delta+1)}u_{\rho}^{2(\delta+1)}(1,t)\right],
%	\nonumber\\&=\frac{1}{2}\frac{d}{dt}\left[\|u_{\rho,x}(t)\|_{\L^2}^2+g_1(u_{\rho}(1,t))u_{\rho}(1,t)\right],
\end{align}
since $u_{\rho,t}(0,t)=0$. Substituting the calculations \eqref{230}-\eqref{232} and \eqref{229}  in  \eqref{241}, and then integrating from $0$ to $t$, we deduce 
\begin{align}\label{233}
	&	\|u_{\rho,x}(t)\|_{\L^2}^2+\frac{\eta}{\nu}u_{\rho}^2(1,t)+\frac{\alpha^2}{\nu\eta(\delta+2)^2(\delta+1)}u_{\rho}^{2(\delta+1)}(1,t)+\nu\int_0^t\|u_{\rho,xx}(s)\|_{\L^2}^2ds\nonumber\\&\quad+2\beta(2\delta+1)\int_0^t\|u_{\rho}^{\delta}(s)u_{\rho,x}(s)\|_{\L^2}^2ds\nonumber\\&\leq \|u_{0,x}\|_{\L^2}^2+\frac{\eta}{\nu}u_{0}^2(1)+\frac{\alpha^2}{\nu\eta(\delta+2)^2(\delta+1)}u_{0}^{2(\delta+1)}(1)\nonumber\\&\quad+\mu \int_0^t\left[\frac{\delta}{\nu^2}\left(\eta+\frac{\alpha^2}{\eta(\delta+2)^2}u_{\rho}^{2\delta}(1,s)\right)^2u(1,s)\right]^2ds+\frac{2\alpha^2}{\nu}\int_0^t\|u_{\rho}(s)\|_{\L^{\infty}}^{2\delta}\|u_{\rho,x}(s)\|_{\L^2}^2ds\nonumber\\&\quad+\frac{2\beta^2(1+\gamma)^2}{\nu}\int_0^t\|u_{\rho}(s)\|_{\L^{2(\delta+1)}}^{2(\delta+1)}ds,
\end{align}
for all $t\in[0,T]$. Note that for the time dependent problem, the estimate \eqref{231} is valid for $1\leq\delta\leq 2$ only.  Application of Gronwall's and Agmon's  inequalities  in \eqref{233} yield
\begin{align}\label{244}
	&	\|u_{\rho,x}(t)\|_{\L^2}^2+\frac{\eta}{\nu}u_{\rho}^2(1,t)+\frac{\alpha^2}{\nu\eta(\delta+2)^2(\delta+1)}u_{\rho}^{2(\delta+1)}(1,t)+\nu\int_0^t\|u_{\rho,xx}(s)\|_{\L^2}^2ds\nonumber\\&\quad+2\beta(2\delta+1)\int_0^t\|u_{\rho}^{\delta}(s)u_{\rho,x}(s)\|_{\L^2}^2ds\nonumber\\&\leq\left\{\|u_{0,x}\|_{\L^2}^2+\frac{1}{\nu(\delta+1)}\left[u_{0}^2(1)+\frac{\alpha^2}{\eta(\delta+2)^2}u_{0}^{2(\delta+1)}(1)\right]+\frac{2\beta^2(1+\gamma)^2}{\nu}\int_0^T\|u_{\rho}(s)\|_{\L^{2(\delta+1)}}^{2(\delta+1)}ds\right\}\nonumber\\&\quad\times\exp\left\{\frac{\delta(\delta+1)}{\nu}\int_0^T\left(\eta u_{\rho}^{2}(1,s)+\frac{\alpha^2}{\eta(\delta+2)^2}u_{\rho}^{2\delta+2}(1,s)\right)ds+\frac{2\alpha^2}{\nu}\int_0^T\|u_{\rho}(s)\|_{\L^{\infty}}^{2\delta}ds\right\}\nonumber\\&\leq C\left(\|u_0\|_{\L^2},\alpha,\beta,\gamma,\delta,\mu,\nu,T\right),
\end{align}
for all $t\in[0,T]$ and $\delta=1$.  Once again Agmon's inequality leads to
\begin{align}\label{247}
	\sup_{t\in[0,T]}\|u_{\rho}(t)\|_{\L^{\infty}}\leq \sup_{t\in[0,T]}\|u_{\rho}(t)\|_{\H^1}^{1/2}\sup_{t\in[0,T]}\|u_{\rho}(t)\|_{\L^2}^{1/2}\leq C,
\end{align}
where $C$ is independent of $\rho$. Therefore, for $\rho\geq C$, $f_{\rho}(y)=f(y)$, and hence $u=u_{\rho}$ is a solution to \eqref{2.3} on $[0,T]$. In this way, one can find the standard existence result $u\in\C\left([0,\infty);\D(\mathscr{A}_{\rho})\right)\cap\C^1\left((0,\infty);\L^2(0,1)\right)\subset \C\left([0,\infty);\H^3(0,1)\right)\cap\C^1\left((0,\infty);\L^2(0,1)\right) $ for the solution of the generalized Korteweg-de Vries-Burgers-Huxley equation \eqref{1.1}-\eqref{1.2} with the boundary control \eqref{1.3}. The case of $\delta=2$ can be established in a similar way. 

Let us now prove the uniqueness for $\delta=1,2$. Let $u_1$ and $u_2$ be two solutions of \eqref{2.3} on $[0,T]$. Then $u=u_1-u_2$ satisfies the following energy equality:
\begin{equation}\label{235}
	\|u(t)\|_{\L^2}^2=\|u_0\|_{\L^2}^2-2\int_0^t\langle\mathscr{A}(u_1)(s)-\mathscr{A}(u_2)(s),u_1(s)-u_2(s)\rangle ds,
\end{equation}
for all $t\in[0,T]$. The term $\langle\mathscr{A}(u_1)-\mathscr{A}(u_2),u_1-u_2\rangle $  can be  estimated similar way as in \eqref{2.11} except for the term $\langle(f(u_1))_x-(f(u_2))_x,u_1-u_2\rangle$. We estimate it using H\"older's, Agmon's and Young's inequalities as 
\begin{align}
	\langle(f(u_1))_x-(f(u_2))_x,u_1-u_2\rangle&=\langle u_1^{\delta}u_{1,x}-u_2^{\delta}u_{2,x},u_1-u_2\rangle\nonumber\\&=\langle (u_1^{\delta}-u_2^{\delta})u_{1,x},u_1-u_2\rangle+\langle u_2^{\delta}(u_{1,x}-u_{2,x}),u_1-u_2\rangle\nonumber\\&\leq \delta\|(\theta u_1+(1-\theta)u_2)^{\delta-1}\|_{\L^{2}}\|u_{1,x}\|_{\L^2}\|u_1-u_2\|_{\L^{\infty}}^2\nonumber\\&\quad+\|u_2\|_{\L^{\infty}}^{\delta}\|u_{1,x}-u_{2,x}\|_{\L^2}\|u_1-u_2\|_{\L^2}\nonumber\\&\leq\frac{\nu}{2}\|u_x\|_{\L^2}^2+\frac{\delta^2}{\nu}\left(\|u_1\|_{\L^{2(\delta-1)}}^{2(\delta-1)}+\|u_2\|_{\L^{2(\delta-1)}}^{2(\delta-1)}\right)\|u_{1,x}\|_{\L^2}^2\|u\|_{\L^2}^2\nonumber\\&\quad+\frac{1}{\nu}\|u_2\|_{\L^{\infty}}^{2\delta}\|u\|_{\L^2}^2. 
\end{align}
Therefore, from \eqref{235}, we infer 
\begin{align}\label{237}
	&	\|u(t)\|_{\L^2}^2+\nu\int_0^t\|u_{x}(s)\|_{\L^2}^2ds+\frac{\beta}{2^{2\delta}}\int_0^t\|u(s)\|_{\L^{2(\delta+1)}}^{2(\delta+1)}ds\nonumber\\&\leq \|u_0\|_{\L^2}^2+2^{2\delta}\beta(1+\gamma)^2(\delta+1)^2\int_0^t\|u(s)\|_{\L^2}^2ds+\frac{2}{\nu}\int_0^t\|u_2(s)\|_{\L^{\infty}}^{2\delta}\|u(s)\|_{\L^2}^2ds\nonumber\\&\quad+\frac{2\delta^2}{\nu}\int_0^t\left(\|u_1(s)\|_{\L^{2(\delta-1)}}^{2(\delta-1)}+\|u_2(s)\|_{\L^{2(\delta-1)}}^{2(\delta-1)}\right)\|u_{1,x}(s)\|_{\L^2}^2\|u(s)\|_{\L^2}^2ds,
\end{align}
for all $t\in[0,T]$. An application of Gronwall's inequality in \eqref{237} provides
\begin{align}\label{238}
	&	\|u(t)\|_{\L^2}^2+\nu\int_0^t\|u_{x}(s)\|_{\L^2}^2ds\nonumber\\&\leq \|u_0\|_{\L^2}^2e^{2^{2\delta}\beta(1+\gamma)^2(\delta+1)^2T}\exp\left(\frac{2}{\nu}\sup_{t\in[0,T]}\|u_2(t)\|_{\L^2}^{2\delta}\int_0^T\|u_2(t)\|_{\H^1}^{2\delta}dt\right)\nonumber\\&\quad\times\exp\left(\frac{2\delta^2}{\nu}\sup_{t\in[0,T]}\left(\|u_1(t)\|_{\L^{2(\delta-1)}}^{2(\delta-1)}+\|u_2(t)\|_{\L^{2(\delta-1)}}^{2(\delta-1)}\right)\int_0^T\|u_{1,x}(t)\|_{\L^2}^2dt\right),
\end{align}
for all $t\in[0,T]$. Since $\delta=1,2$, the right hand side of \eqref{238} is finite even for $u_0\in\L^2(0,1)$ and the uniqueness of solutions follows since $u_0=0$. 
	\end{proof}

\section{Stabilization}\label{sec3}\setcounter{equation}{0}    The aim of this section  is to establish the $\L^2$-, $\H^1$-  and pointwise exponential stabilization results for the generalized KdV-Burgers-Huxley equation \eqref{1.1}-\eqref{1.2} subject to the controls \eqref{1p3} and \eqref{1.3}.
\begin{theorem}\label{thm2.3}
	For $$\nu>\frac{\beta}{4}(1-\gamma)^2,$$ 	the generalized KdV-Burgers-Huxley equation \eqref{1.1}-\eqref{1.2} is exponentially stable in $\L^2(0,1)$ under the control laws \eqref{1.3}, \eqref{1p3},  respectively for $\delta=1,2$. 
\end{theorem}
\begin{proof}
	Let $\mathcal{V}(\cdot)$ be a  Lyapunov function defined by
	\begin{align}\label{eq2}
		\mathcal{V}(t)=\frac{1}{2}\int_0^1u^2(x,t)dx,\ t\geq 0. 
	\end{align}
	Taking the derivative with respect to $t$, performing an integration by parts and then using \eqref{215} and \eqref{2k12}, for $\delta=1,2$, respectively,  we deduce 
	\begin{align}\label{eq1}
		\dot{\mathcal{V}}(t)&=\int_0^1u(x,t)u_t(x,t)dx\nonumber\\&=\int_0^1u(x,t)[\nu u_{xx}(x,t)-\mu u_{xxx}-\alpha u^{\delta}(x,t)u_x(x,t)+\beta u(x,t)(1-u^{\delta}(x,t))(u^{\delta}(x,t)-\gamma)]dx\nonumber\\&=-\nu\int_0^1u_x^2(x,t)dx+\nu u(1,t)u_x(1,t)-\mu u(1,t)u_{xx}(1,t)+\frac{\mu}{2} u_x^2(1,t)-\frac{\mu}{2} u_x^2(0,t)\nonumber\\&\quad-\frac{\alpha}{\delta+2}u^{\delta+2}(1,t)+\beta(1+\gamma)\int_0^1u^{\delta+2}(x,t)dx-\beta\gamma\int_0^1u^2(x,t)dx-\beta\int_0^1u^{2(\delta+1)}(x,t)dx\nonumber\\&\leq-2\left(\nu-\frac{\beta}{4}(1-\gamma)^2\right)\mathcal{V}(t),
	\end{align}
for all $t>0$. Thus, we deduce  $\|u(t)\|_{\L^2}\leq \|u_0\|_{\L^2}e^{-\zeta t},$ for all $t\geq 0$ and since $\zeta=\left(\nu-\frac{\beta}{4}(1-\gamma)^2\right)>0$,  $\|u(t)\|_{\L^2}$ converges to zero exponentially as $t\to\infty$. 
\end{proof}

By controlling the convective term using  diffusion and reaction terms, under further restrictions on $\nu$,  for any $\delta\in[1,\infty)$, one can obtain $\L^2$-stabilization of generalized KdV-Burgers-Huxley equation by using  an another control law.
\begin{theorem}\label{thm2.4}
	For any $\delta\in[1,\infty)$ and for \begin{align}\label{2p55}\nu>\max\left\{\frac{\alpha^2}{2\beta(2\delta+1)},\frac{\beta}{2}(1+\gamma^2)+\frac{\alpha^2}{2\beta(\delta+2)^{\frac{2(\delta+1)}{\delta+2}}}\right\},\end{align}	the generalized KdV-Burgers-Huxley equation \eqref{1.1}-\eqref{1.2} is exponentially stable in $\L^2(0,1)$ under the control law
	\begin{equation}\label{256}
		u(0,t)=0,\  u_x(1,t)=- u(1,t),\  u_{xx}(1,t)=u(1,t).
	\end{equation}
\end{theorem}
\begin{proof}
For $\nu>\frac{\alpha^2}{2\beta(\delta+2)^{\frac{2(\delta+1)}{\delta+2}}}$, the well-posedness of the problem \eqref{1.1}-\eqref{1.2} for  any $\delta\in[1,\infty)$ under the control law \eqref{256} is not difficult. 	Note that the control given in \eqref{256} does not cause any additional difficulty in establishing the estimate \eqref{247}. Using Lemma \ref{lema3}, one can estimate $-\frac{\alpha}{\delta+2}u^{\delta+2}(1)$ as 
	\begin{align}
		-\frac{\alpha}{\delta+2}u^{\delta+2}(1)&\leq \frac{\alpha}{\delta+2}\|u\|_{\L^{\infty}}^{\delta+2}\leq \frac{\alpha}{\delta+2}\|u_x\|_{\L^2}\|u\|_{\L^{2(\delta+1)}}^{\delta+1}\nonumber\\&\leq\frac{\theta\beta}{2}\|u\|_{\L^{2(\delta+1)}}^{2(\delta+1)}+\frac{\alpha^2}{2\theta\beta(\delta+2)^{\frac{2(\delta+1)}{\delta+2}}}\|u_x\|_{\L^2}^2,
	\end{align}
	for some $0<\theta\leq 1$.	But we know that 
	\begin{align}
		\beta(1+\gamma)(u^{\delta+1},u)\leq\beta(1+\gamma)\|u\|_{\L^{2(\delta+1)}}^{\delta+1}\|u\|_{\L^2}\leq\frac{\theta\beta}{2}\|u\|_{\L^{2(\delta+1)}}^{2(\delta+1)}+\frac{\beta(1+\gamma)^2}{2\theta}\|u\|_{\L^2}^2. 
	\end{align}
	 In fact, for the control given in \eqref{256},  the equality  \eqref{217} yields 
	\begin{align}
		\langle\mathscr{A}(u),u\rangle\geq\left(\nu-\frac{\alpha^2}{2\beta(\delta+2)^{\frac{2(\delta+1)}{\delta+2}}}\right)\|u_x\|_{\L^2}^2+\left(\nu+\frac{\mu}{2}\right)u^2(1)-\frac{\beta(1+\gamma^2)}{2}\|u\|_{\L^2}^2. 
	\end{align}
	Therefore, a calculation similar to \eqref{239} yields 
	\begin{align}\label{260}
		&\|u(t)\|_{\L^2}^2+2\left(\nu-\frac{\alpha^2}{2\theta\beta(\delta+2)^{\frac{2(\delta+1)}{\delta+2}}}\right)\int_0^t\|u_x(s)\|_{\L^2}^2ds+2\left(\nu+\frac{\mu}{2}\right)\int_0^tu^2(1,s)ds\nonumber\\&\quad+2\beta(1-\theta)\int_0^t\|u(s)\|_{\L^{2(\delta+1)}}^{2(\delta+1)}ds\nonumber\\&\leq \|u_0\|_{\L^2}^2+\frac{\beta(1+\gamma^2)}{\theta}\int_0^t\|u(s)\|_{\L^2}^2ds,
	\end{align}
	for all $t\in[0,T]$ and some $0<\theta<1$. Therefore, for $\nu>\frac{\alpha^2}{2\beta(\delta+2)^{\frac{2(\delta+1)}{\delta+2}}}$, an application of Gronwall's inequality yields 
	\begin{align}\label{39}
		&\|u(t)\|_{\L^2}^2+2\left(\nu-\frac{\alpha^2}{2\theta\beta(\delta+2)^{\frac{2(\delta+1)}{\delta+2}}}\right)\int_0^t\|u_x(s)\|_{\L^2}^2ds+2\left(\nu+\frac{\mu}{2}\right)\int_0^tu^2(1,s)ds\nonumber\\&\quad+2\beta(1-\theta)\int_0^t\|u(s)\|_{\L^{2(\delta+1)}}^{2(\delta+1)}ds\leq \|u_0\|_{\L^2}^2e^{\frac{\beta(1+\gamma^2)T}{\theta}}, 
	\end{align}
	for all $t\in[0,T]$ and some $0<\theta<1$. 
	
	Let us now establish an estimate similar to \eqref{244}. Calculations similar to \eqref{229} and \eqref{230} provide
	\begin{align*}
		-	\int_0^1u_{t}(t,x)&u_{xx}(t,x)dx=\frac{1}{2}\frac{d}{dt}\left[\|u_{x}(t)\|_{\L^2}^2+u^2(1,t)\right],
	\end{align*}
	and 
	\begin{align*}
		(\mathscr{A}(u),-u_{xx})&\geq\nu\|u_{xx}\|_{\L^2}^2-\frac{\mu}{2}u^2(1)-\alpha(u^{\delta}u_{x},u_{xx})+\beta(1+\gamma)(u^{\delta+1},u_{xx})\nonumber\\&\quad+\beta\gamma\|u_{x}\|_{\L^2}^2+\beta u^2(1)+\beta(2\delta+1)\|u^{\delta}u_{x}\|_{\L^2}^2+\beta u^{2\delta+2}(1).
	\end{align*}
	We calculate the terms $-\alpha(u^{\delta}u_{x},u_{xx})$ and $\beta(1+\gamma)(u^{\delta+1},u_{xx})$  as 
	\begin{align*}
		-\alpha(u^{\delta}u_{x},u_{xx})&\leq\alpha\|u^{\delta}u_{x}\|_{\L^2}\|u_{xx}\|_{\L^2}\leq\frac{\theta\nu}{2}\|u_{xx}\|_{\L^2}^2+\frac{\alpha^2}{2\theta\nu}\|u^{\delta}u_x\|_{\L^2}^2,\\
		\beta(1+\gamma)(u^{\delta+1},u_{xx})&\leq \beta(1+\gamma)\|u\|_{\L^{2(\delta+1)}}^{\delta+1}\|u_{xx}\|_{\L^2}\leq\frac{\theta\nu}{2}\|u_{xx}\|_{\L^2}^2+\frac{\beta^2(1+\gamma)^2}{2\theta\nu}\|u\|_{\L^{2(\delta+1)}}^{2(\delta+1)},
	\end{align*}
	for some $0<\theta\leq 1$. Substituting the above estimates in \eqref{241}, we deduce 
	\begin{align}\label{262}
		&\|u_{x}(t)\|_{\L^2}^2+u^2(1,t)+2(1-\theta)\nu\int_0^t\|u_{xx}(s)\|_{\L^2}^2ds+2\beta\int_0^tu^{2(\delta+1)}(1,s)ds\nonumber\\&\quad+2\left(\beta(2\delta+1)-\frac{\alpha^2}{2\theta\nu}\right)\int_0^t\|u^{\delta}(s)u_{x}(s)\|_{\L^2}^2ds\nonumber\\&\leq\|u_{0,x}\|_{\L^2}^2+u^2_0(1)+\mu\int_0^tu^2(1,s)ds+\frac{\beta^2(1+\gamma)^2}{\theta\nu}\int_0^t\|u(s)\|_{\L^{2(\delta+1)}}^{2(\delta+1)}ds,
	\end{align}
	for all $t\in[0,T]$ and some $0<\theta<1$. For $\nu>\frac{\alpha^2}{2\beta(2\delta+1)}$, an application of Gronwall's and Agmon's inequalities in \eqref{262} yields 
	\begin{align}
		\|u_{x}(t)\|_{\L^2}^2\leq\left\{\|u_{0,x}\|_{\L^2}^2+2\|u_0\|_{\L^2}\|u_{0,x}\|_{\L^2}+\frac{\beta^2(1+\gamma)^2}{\theta\nu}\int_0^T\|u(s)\|_{\L^{2(\delta+1)}}^{2(\delta+1)}ds,\right\}e^{\mu T},
	\end{align}
	for all $t\in[0,T]$. Therefore an application of \eqref{39} and Agmon's inequality imply \eqref{247} holds for all $\delta\in[1,\infty)$ and the proof of existence is completed. The uniqueness of strong solutions follow immediately from the estimate \eqref{238} as $u\in\L^{\infty}((0,T)\times(0,1))$. 
	
	Finally, for the control law given in \eqref{256}, we infer from \eqref{eq1} under the condition \eqref{2p55} that 
	\begin{align}\label{257}
		\dot{\mathcal{V}}(t)&\leq-2\left(\nu-\frac{\beta}{2\theta}(1+\gamma^2)-\frac{\alpha^2}{2\theta\beta(\delta+2)^{\frac{2(\delta+1)}{\delta+2}}}\right)\mathcal{V}(t),
	\end{align}
	for all $t>0$ and some $0<\theta<1$. Therefore, one can obtain the exponential stabilization in $\L^2(0,1)$ for any $\delta\in[1,\infty)$. 
\end{proof}

Under the condition \eqref{2p55}, let us establish the $\H^1$-stabilization of generalized KdV-Burgers-Huxley equation with the control law \eqref{256}.

\begin{theorem}\label{thm3.8}
Under the condition \eqref{2p55}, 	global exponential stability in the $\H^1$-norm 	sense holds for the generalized KdV-Burgers-Huxley equation \eqref{1.1}-\eqref{1.2} under the control law \eqref{256}, that is, 
	\begin{align}\label{3p17}
		\|u(t)\|_{\H^1}^2\leq 	2\|u_x(t)\|_{\L^2}^2\leq 2\left\{\|u_{0,x}\|_{\L^2}^2+u_0^2(1)+\left[\frac{1}{\varrho}(\varrho+2(\mu+1))+\frac{\beta(1+\gamma)^2}{2\theta(1-\theta)\nu}\right]\|u_0\|_{\L^2}^2\right\}e^{-\frac{\varrho}{2} t},
	\end{align}
	for all $t\geq 0$ and $u_0\in\H^1(0,1)$, where 
	\begin{align}\label{3p18}
		\varrho=\left(\nu-\frac{\alpha^2}{2\theta\beta(\delta+2)^{\frac{2(\delta+1)}{\delta+2}}}-\frac{\beta(1+\gamma^2)}{2\theta}\right)>0,
	\end{align}
	for some $0<\theta<1$. 
\end{theorem}
\begin{proof}
	A calculation similar to \eqref{262} yields 
	\begin{align*}
		&\frac{d}{dt}\Psi(t)+2(1-\theta)\nu\|u_{xx}(t)\|_{\L^2}^2+2\beta u^{2(\delta+1)}(1,t)+2\left(\beta(2\delta+1)-\frac{\alpha^2}{2\theta\nu}\right)\|u^{\delta}(t)u_{x}(t)\|_{\L^2}^2\nonumber\\&\leq \mu u^2(1,t)+\frac{\beta^2(1+\gamma)^2}{\theta\nu}\|u(t)\|_{\L^{2(\delta+1)}}^{2(\delta+1)},
	\end{align*}
	where $\Psi(t)=\|u_{x}(t)\|_{\L^2}^2+u^2(1,t)$ for a.e. $t\in[0,T]$. Let us multiply the above equation by $e^{\frac{\varrho}{2} t}$ to deduce 
	\begin{align*}
		&\frac{d}{dt}[e^{\frac{\varrho}{2} t}\Psi(t)]+2(1-\theta)\nu e^{\frac{\varrho}{2} t}\|u_{xx}(t)\|_{\L^2}^2+2\beta e^{\frac{\varrho}{2} t} u^{2(\delta+1)}(1,t)\nonumber\\&\quad+2\left(\beta(2\delta+1)-\frac{\alpha^2}{2\theta\nu}\right)e^{\frac{\varrho}{2} t}\|u^{\delta}(t)u_{x}(t)\|_{\L^2}^2\nonumber\\&\leq\frac{\varrho}{2}e^{\frac{\varrho}{2} t}\Psi(t)+ \mu e^{\frac{\varrho}{2} t} u^2(1,t)+\frac{\beta^2(1+\gamma)^2}{\theta\nu}e^{\frac{\varrho}{2} t}\|u(t)\|_{\L^{2(\delta+1)}}^{2(\delta+1)}\nonumber\\&\leq e^{\frac{\varrho}{2} t}\left[(\varrho+2(\mu+1))\|u_{x}(t)\|_{\L^2}^2+\frac{\beta^2(1+\gamma)^2}{\theta\nu}\|u(t)\|_{\L^{2(\delta+1)}}^{2(\delta+1)}\right],
	\end{align*}
	where $\varrho$ is defined in \eqref{3p18}, and we have used the fact that $u^2(1)\leq\|u\|_{\L^{\infty}}^2\leq 2\|u\|_{\L^2}\|u_x\|_{\L^2}\leq 2\|u_x\|_{\L^2}^2$. Integrating the above inequality from $0$ to $t$, we find 
	\begin{align}\label{264}
		e^{\frac{\varrho}{2} t}\Psi(t)\leq\Psi(0)+\int_0^te^{\frac{\varrho}{2} s}\left[(\varrho+2(\mu+1))\|u_{x}(s)\|_{\L^2}^2+\frac{\beta^2(1+\gamma)^2}{\theta\nu}\|u(s)\|_{\L^{2(\delta+1)}}^{2(\delta+1)}\right]ds,
	\end{align}
	for all $t\geq 0$. 
	
	A calculation similar to \eqref{260} yields 
	\begin{align*}
		&\frac{d}{dt}\|u(t)\|_{\L^2}^2+2\left(\nu-\frac{\alpha^2}{2\theta\beta(\delta+2)^{\frac{2(\delta+1)}{\delta+2}}}-\frac{\beta(1+\gamma^2)}{2\theta}\right)\|u_x(t)\|_{\L^2}^2+2\beta(1-\theta)\|u(t)\|_{\L^{2(\delta+1)}}^{2(\delta+1)}\leq 0,
	\end{align*}
	for a.e. $t\in[0,T]$. Therefore, by an application of the Poincar\'e inequality and variation of constants formula in the above inequality implies, for $\varrho=\left(\nu-\frac{\alpha^2}{2\theta\beta(\delta+2)^{\frac{2(\delta+1)}{\delta+2}}}-\frac{\beta(1+\gamma^2)}{2\theta}\right)>0$,
	\begin{align}\label{265}
		\|u(t)\|_{\L^2}^2+\varrho\int_0^te^{\varrho s}\|u_x(s)\|_{\L^2}^2ds+2\beta(1-\theta)\int_0^te^{\varrho s}\|u(s)\|_{\L^{2(\delta+1)}}^{2(\delta+1)}ds\leq\|u_0\|_{\L^2}^2, 
	\end{align}
	for all $t\geq 0$. Substituting \eqref{265} in \eqref{264}, we obtain our required result as
	\begin{align}
		\Psi(t)\leq\left\{\Psi(0)+\left[\frac{1}{\varrho}(\varrho+2(\mu+1))+\frac{\beta(1+\gamma)^2}{2\theta(1-\theta)\nu}\right]\|u_0\|_{\L^2}^2\right\}e^{-\frac{\varrho}{2} t},
	\end{align}
	for some $0<\theta<1$. Thus the exponential stabilization result in $\H^1(0,1)$ (see \eqref{3p17})  follows immediately. 
\end{proof}

Finally, an application of Agmon’s inequality (Lemma \ref{lem21}) yields the following result on the pointwise exponential convergence:
\begin{theorem}\label{thm3.10}
	Consider the generalized KdV-Burgers-Huxley equation \eqref{1.1}-\eqref{1.2} under the control law \eqref{256}.  If $u_0\in\H^1(0,1)$ and $\nu$ satisfies the condition \eqref{2p55}, then we have 
	\begin{align}
		\max_{x\in[0,1]}|u(x,t)|\leq 4\left\{\|u_{0,x}\|_{\L^2}^2+u_0^2(1)+\left[\frac{1}{\varrho}(\varrho+2(\mu+1))+\frac{\beta(1+\gamma)^2}{2\theta(1-\theta)\nu}\right]\|u_0\|_{\L^2}^2\right\}e^{-\frac{\varrho}{2} t},
	\end{align}
	for all $t\geq 0$,	where $\varrho$ is defined in \eqref{3p18}. 
\end{theorem}

%\subsection{Adaptive control} In this subsection, we consider the adaptive control problem of the following  damped gneralized Burgers-Huxley equation subject to \eqref{2} (that is, when $\nu>0$ and $\mu>0$ are unknown):
%\begin{proof}
%	A calculation similar to \eqref{eq1}  yields 
%	\begin{align}
%		\dot{V}(t)&=-\nu\int_0^1u_x^2(x,t)dx+\nu u(1,t)u_x(1,t)-\mu u(1,t)u_{xx}(1,t)+\frac{\mu}{2} u_x^2(1,t)-\frac{\mu}{2} u_x^2(0,t)\nonumber\\&\quad-\frac{\alpha}{\delta+2}u^{\delta+2}(1,t)+\beta(1+\gamma)\int_0^1u^{\delta+2}(x,t)dx-\beta\gamma\int_0^1u^2(x,t)dx-\beta\int_0^1u^{2(\delta+1)}(x,t)dx
%	\end{align}
%\end{proof}

\section{Numerical Results}\label{sec4}\setcounter{equation}{0}  
In our numerical studies of the third-order differential system \eqref{1.1}-\eqref{1.2}, we  implement a modified Chebyshev Collocation Method given by Kosloff and Tal-Ezer \cite{KTAL} to ensure stability in time steps. This method derives modified grids, which are then transformed to fit within our desired $[0,1]$ interval. Consequently, we compute the corresponding modified Chebyshev differentiation matrices, following the approach discussed in \cite{RRY}. For spatial discretization, these matrices are applied over the transformed Chebyshev nodes, and we utilize the backward Euler method for time discretization.  A MATLAB program is developed in accordance with these methodologies and effectively solved both uncontrolled and controlled versions of the Generalized Korteweg-de Vries-Burgers-Huxley (GKdVBH) equation \eqref{1.1}-\eqref{1.2} and controls outlined in \eqref{1.3}, \eqref{1p3}, \eqref{256}  and Remark \ref{rem1.1}. Over transformed nodes, we apply the Gauss-Lobatto quadrature for numerical integration, and thus, the $L^{2}$-norm of the solution $u(\cdot,t)$ at each time step is plotted against time $t$.

\begin{figure}[h]
	\centering
	% First image
	\begin{minipage}{0.68\textwidth}
		\centering
		\includegraphics[width=\linewidth]{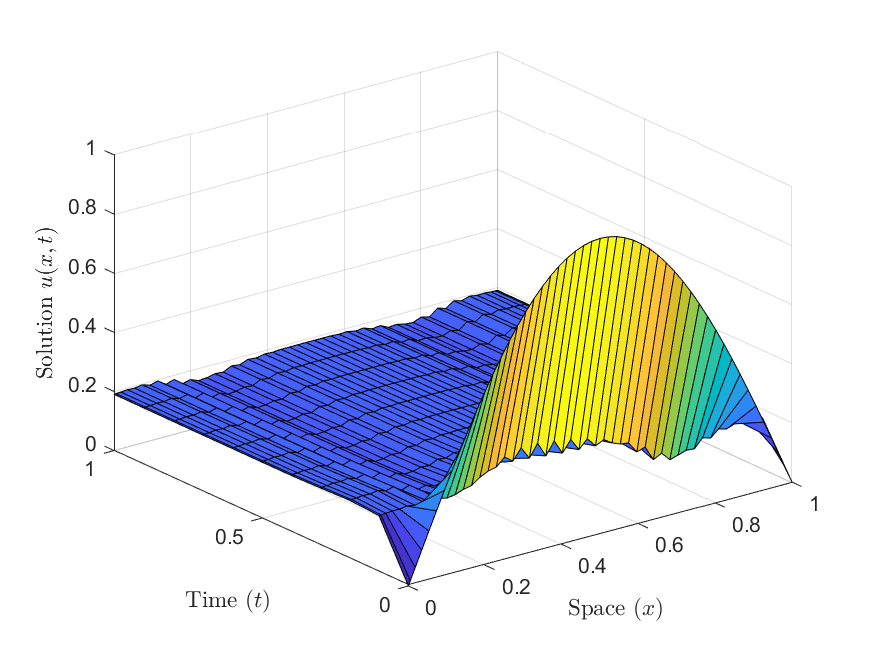}
	\end{minipage}\hfill
	% Second image
	\begin{minipage}{0.28\textwidth}
		\centering
		\includegraphics[width=\linewidth]{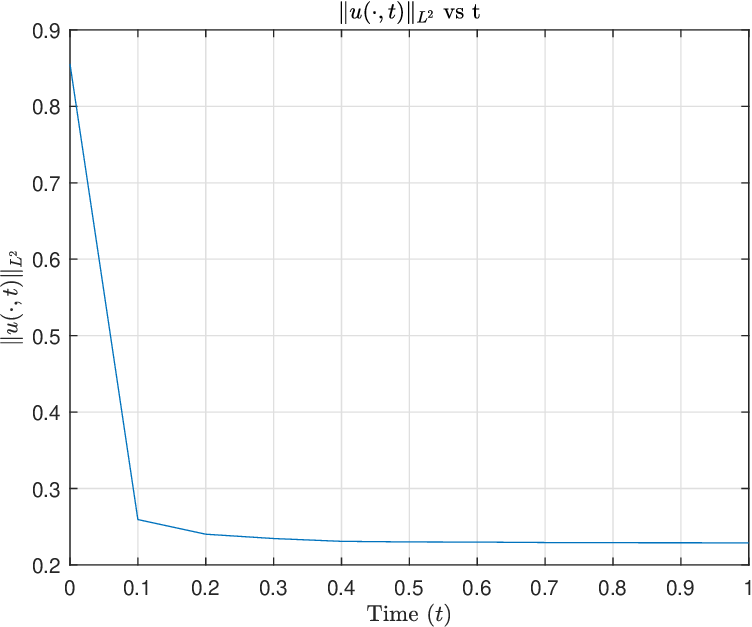} 
	\end{minipage}
	\caption{Time evolution of uncontrolled GKdVBH equation for $\nu = 1, \alpha=1,\beta=1,\delta=1,\gamma=0.5$  and  $u_{0}(x) = \sin(\pi x)$.}
	\label{fig1}
\end{figure}
\begin{figure}[h]
	\centering
	% First image
	\begin{minipage}{0.68\textwidth}
		\centering
		\includegraphics[width=\linewidth]{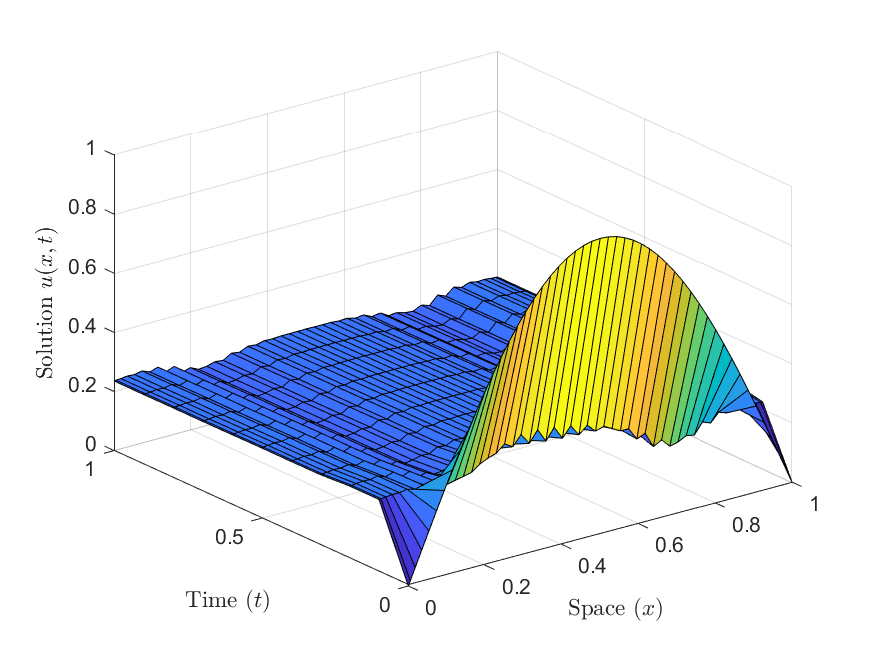}
	\end{minipage}\hfill
	% Second image
	\begin{minipage}{0.28\textwidth}
		\centering
		\includegraphics[width=\linewidth]{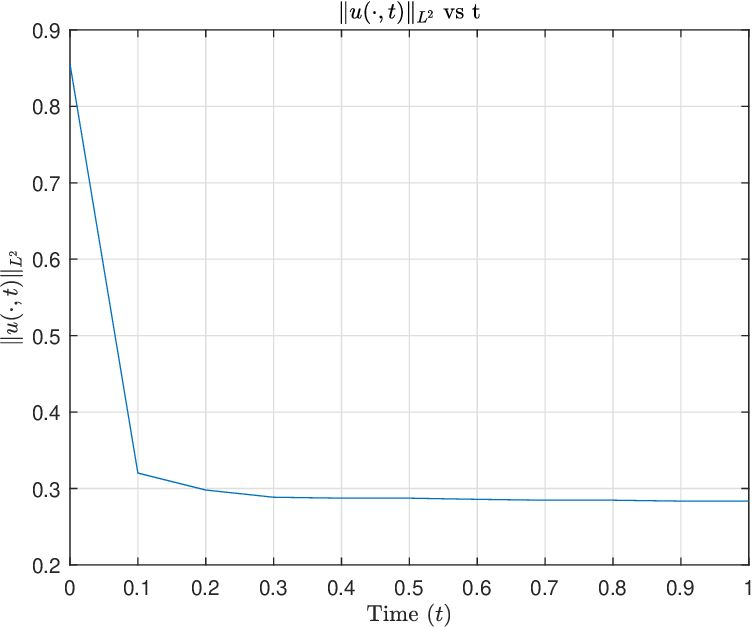} 
	\end{minipage}
	\caption{Time evolution of uncontrolled GKdVBH equation for $\nu = 1, \alpha=1,\beta=1,\delta=2,\gamma=0.5$  and  $u_{0}(x) = \sin(\pi x)$.}
	\label{fig2}
\end{figure}
\begin{figure}[h]
	\centering
	% First image
	\begin{minipage}{0.68\textwidth}
		\centering
		\includegraphics[width=\linewidth]{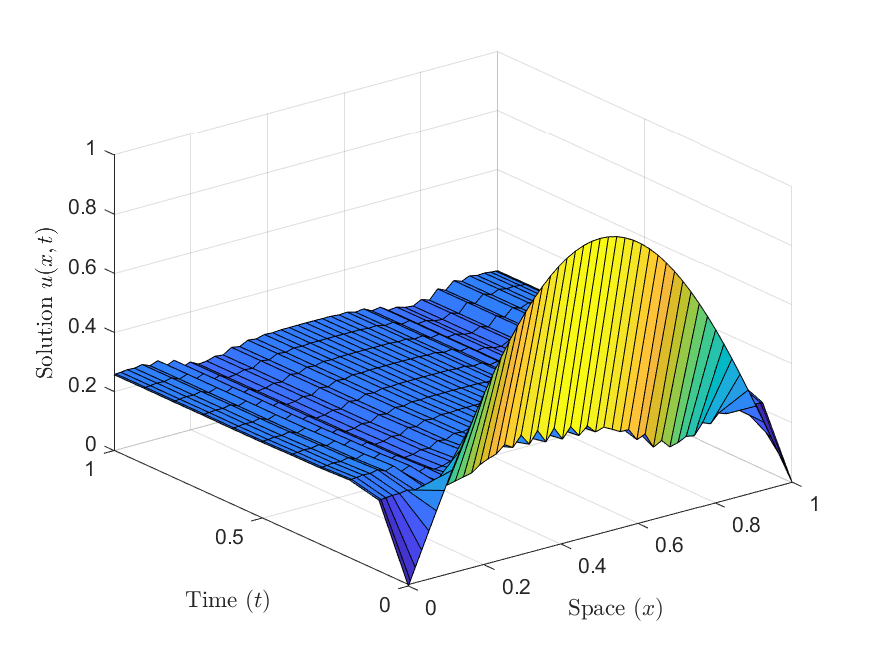}
	\end{minipage}\hfill
	% Second image
	\begin{minipage}{0.28\textwidth}
		\centering
		\includegraphics[width=\linewidth]{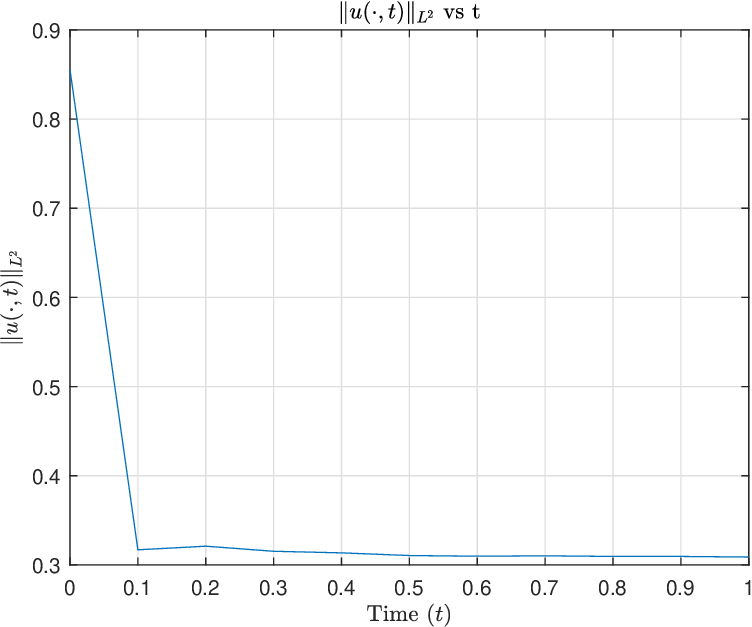} 
	\end{minipage}
	\caption{Time evolution of uncontrolled GKdVBH equation for $\nu = 1, \alpha=1,\beta=1,\delta=3,\gamma=0.5$  and  $u_{0}(x) = \sin(\pi x)$.}
	\label{fig3}
\end{figure}

Figure \ref{fig1}, \ref{fig2} and \ref{fig3}  depicts the temporal evolution of the solution \( u(x,t) \) in the absence of control, characterized by the parameters \(\alpha=1\), \(\beta=1\), \(\gamma=0.5\), \(\nu = 1 \), \(\mu =0.1 \), and the initial condition \( u(x,0) = \sin(\pi x) \) for  \(\delta = 1\),  \(\delta = 2\) and  \(\delta = 3\) respectively. When the control delineated in \eqref{1.3} (\(\delta = 1\)) is implemented with a parameter \(\eta =1\), figure \ref{fig4} elucidates that the solution \( u(x,t) \) converges towards the zero state. Similarly, when a control law \eqref{1p3} (\(\delta = 2\)) is applied with \(\eta =1\), figure \ref{fig5} demonstrates that \( u(x,t) \) approaches the desired quiescent state. These observations validate Theorem \ref{thm2.3}.

\begin{figure}[h]
	\centering
	% First image
	\begin{minipage}{0.68\textwidth}
		\centering
		\includegraphics[width=\linewidth]{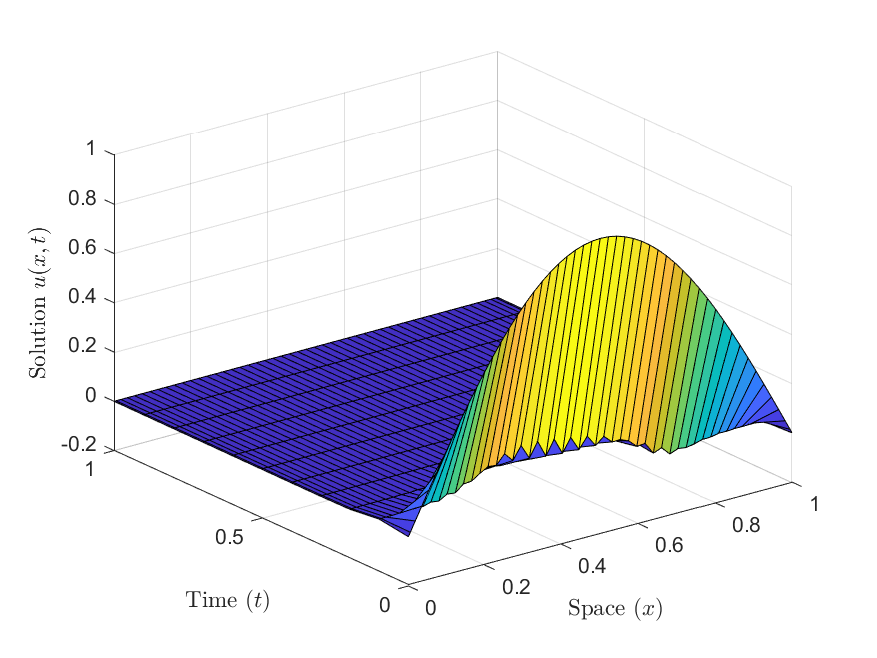}
	\end{minipage}\hfill
	% Second image
	\begin{minipage}{0.28\textwidth}
		\centering
		\includegraphics[width=\linewidth]{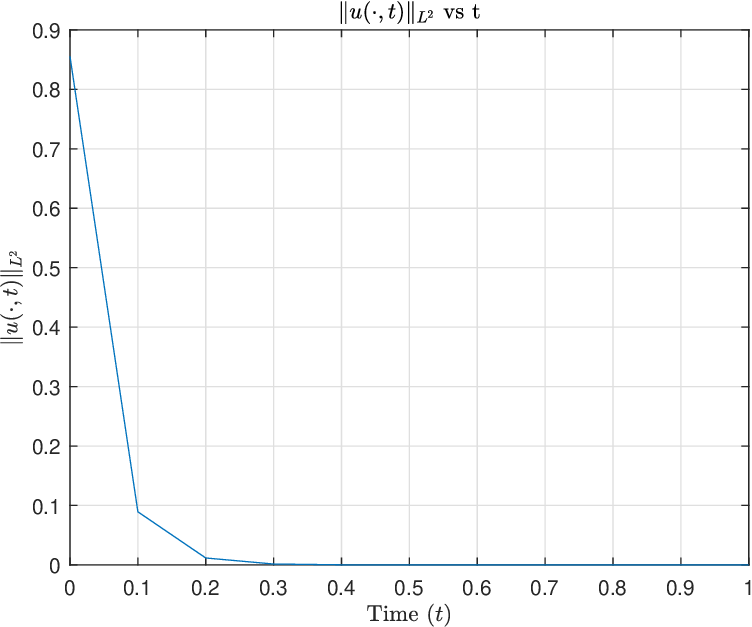} 
	\end{minipage}
	\caption{Time evolution of controlled GKdVBH equation for $\nu = 1, \alpha=1,\beta=1,\delta=1,\gamma=0.5,\eta = 1, \mu = 0.1 $ and $u_{0}(x) = \sin(\pi x)$.}
	\label{fig4}
\end{figure}

\begin{figure}[h]
	\centering
	% First image
	\begin{minipage}{0.68\textwidth}
		\centering
		\includegraphics[width=\linewidth]{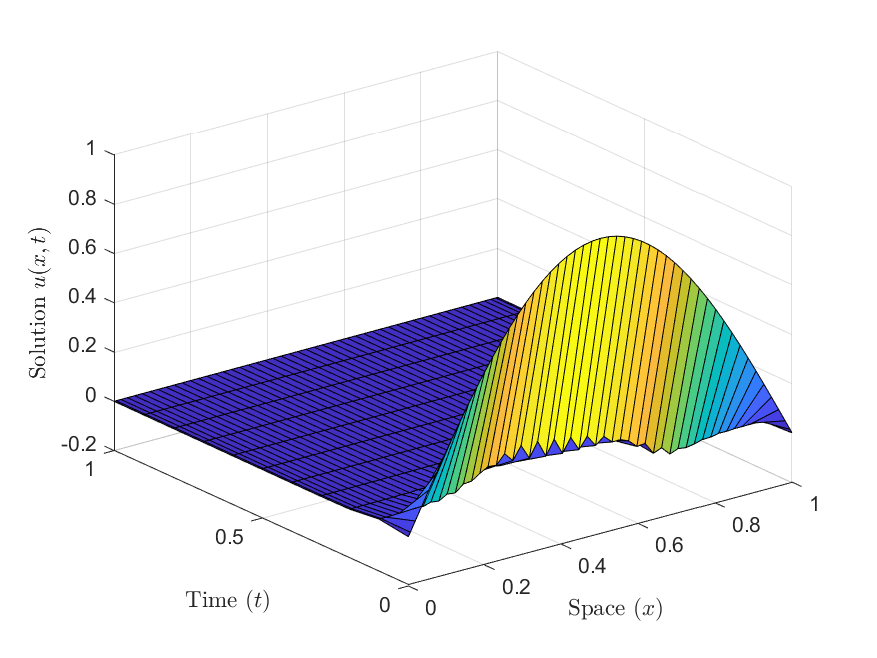}
	\end{minipage}\hfill
	% Second image
	\begin{minipage}{0.28\textwidth}
		\centering
		\includegraphics[width=\linewidth]{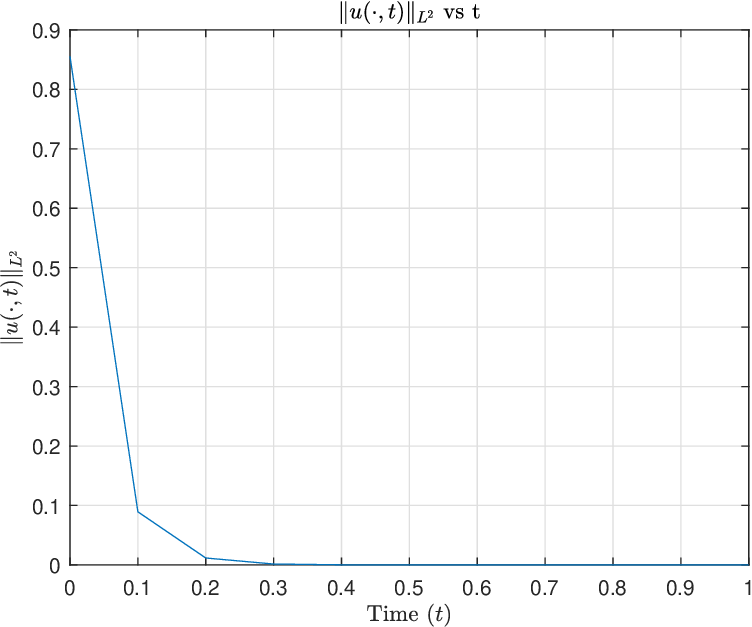} 
	\end{minipage}
	\caption{Time evolution of uncontrolled GKdVBH equation for $\nu = 1, \alpha=1,\beta=1,\delta=2,\gamma=0.5,\eta = 1, \mu = 0.1 $ and $u_{0}(x) = \sin(\pi x)$.}
	\label{fig5}
\end{figure}

Figures \ref{fig6} and \ref{fig7} describe the asymptotic stability achieved by implementing the controls as outlined in Remark \ref{rem1.1}. Specifically, along with condition \eqref{17}, defined by the function $g(\cdot)$ are $\frac{1}{\mu}\left(\eta+\frac{\alpha^2}{\eta(\delta+2)^2}u^{2\delta}(1,t)\right)u(1,t) \ (\text{for }\ \delta=1) \ \text{ and }\ \frac{\eta}{\mu}u(1,t) \ (\text{for }\ \delta=2)$ respectively.

Figure \ref{fig8} validates Theorem \ref{thm2.4}, which shows exponential stability of the control defined in \eqref{256}  $(\text{for }\ \delta=3)$.

Moreover, in studying the convergence rates for controls, the logarithms of the norms plotted over time offer enhanced visual clarity. Linear segments in these plots are indicative of exponential decay, with the slope of these segments directly corresponding to the rate of decay. A less steep slope signifies a slower rate of decay. Note that we have taken $\nu = 1$ so that condition \eqref{2p55} can be met along with the condition in Theorem \ref{thm2.3}, thus, facilitating proper comparison. Figure \ref{fig9} and  \ref{fig10} conclusively demonstrates that the convergence rates for control \eqref{17} is notably slower than controls \eqref{256} and \eqref{1.3} $\text{for }\ \delta=1 \ \text{and }\ \delta=2,$ respectively, corroborating the observations made in Remark \ref{rem1.1}. Figure \ref{fig11} shows that $\max\limits_{x \in [0,1]} |u(x,t)|$ decays to zero, as in accordance with Theorem \ref{thm3.10}.
\begin{figure}[h]
	\centering
	% First image
	\begin{minipage}{0.68\textwidth}
		\centering
		\includegraphics[width=\linewidth]{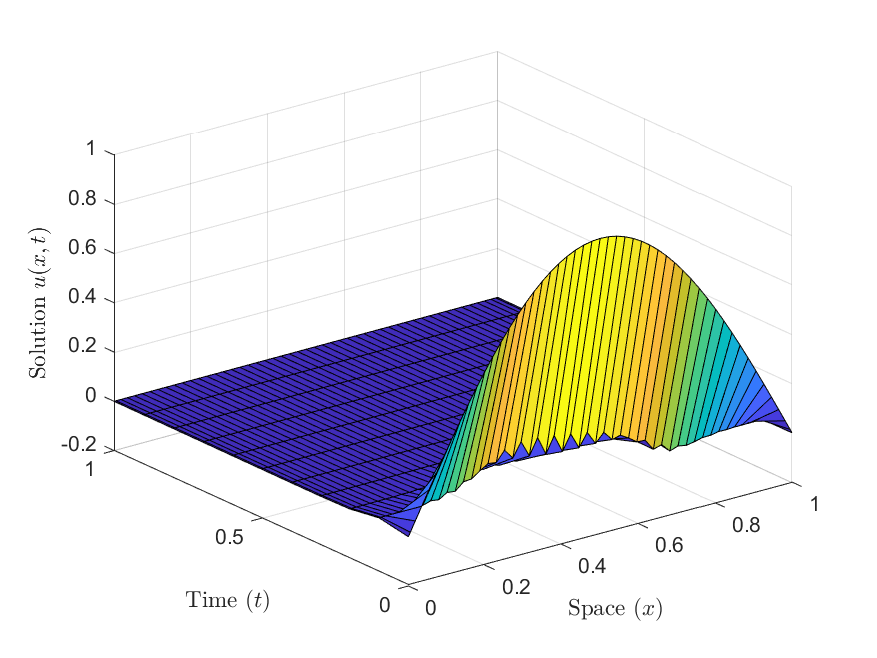}
	\end{minipage}\hfill
	% Second image
	\begin{minipage}{0.28\textwidth}
		\centering
		\includegraphics[width=\linewidth]{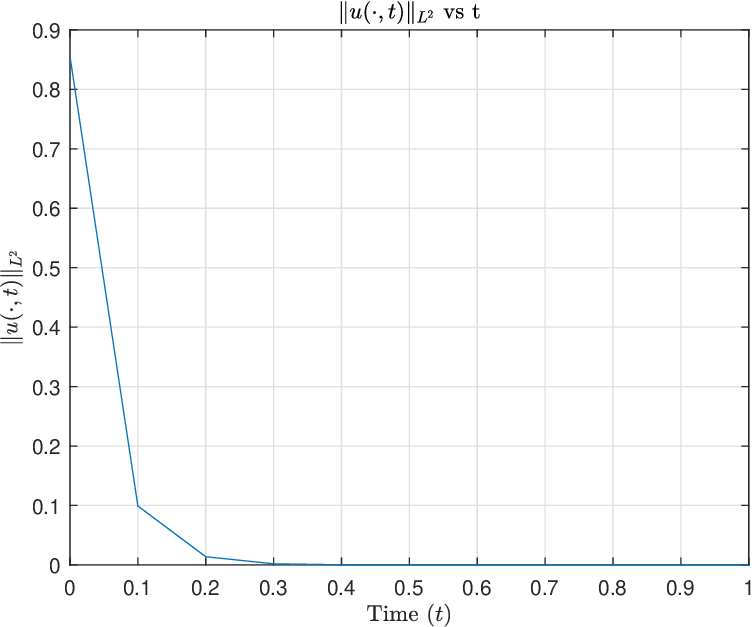} 
	\end{minipage}
	\caption{Time evolution of GKdVBH equation under control defined in \eqref{17} for $\nu = 1, \alpha=1,\beta=1,\delta=1,\gamma=0.5,\eta = 1, \mu = 0.1 $ and $u_{0}(x) = \sin(\pi x)$.}
	\label{fig6}
\end{figure}
\begin{figure}[h]
	\centering
	% First image
	\begin{minipage}{0.68\textwidth}
		\centering
		\includegraphics[width=\linewidth]{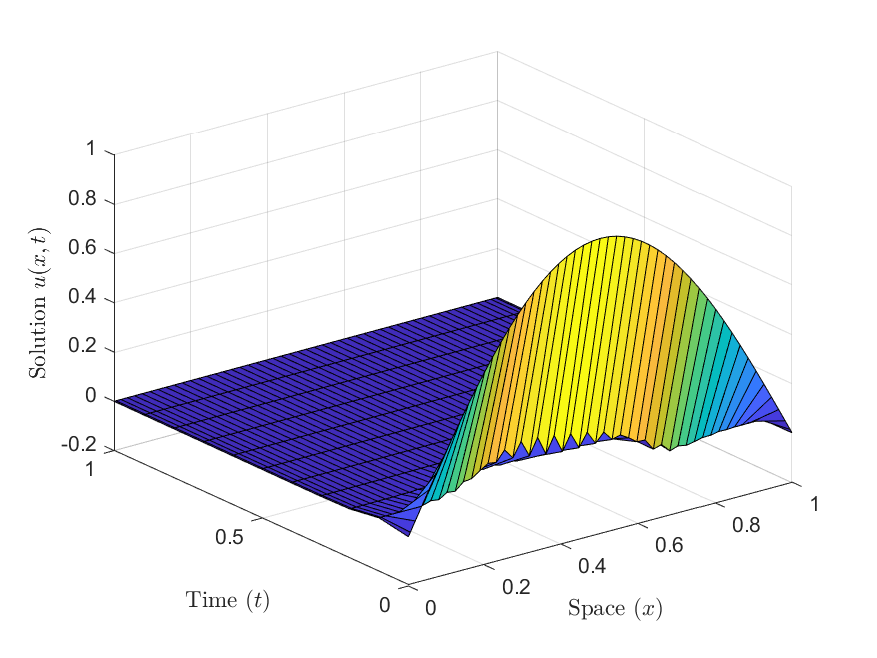}
	\end{minipage}\hfill
	% Second image
	\begin{minipage}{0.28\textwidth}
		\centering
		\includegraphics[width=\linewidth]{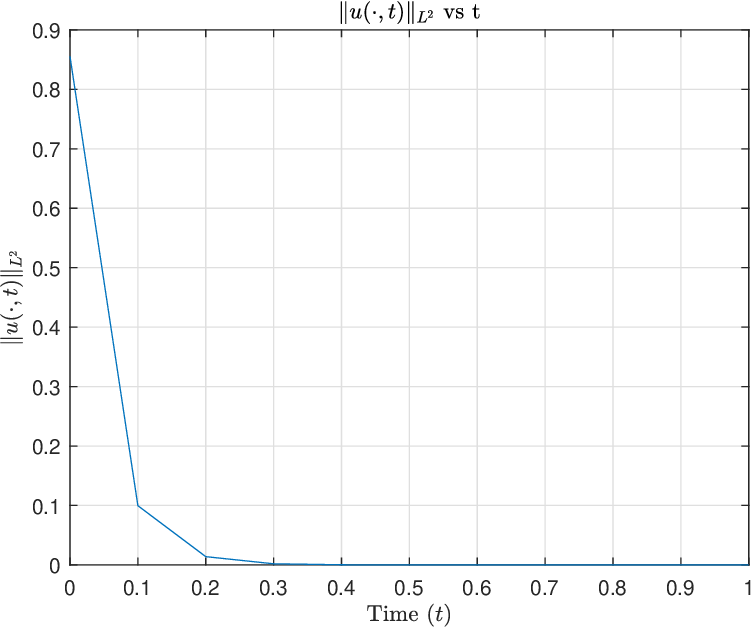} 
	\end{minipage}
	\caption{Time evolution of GKdVBH equation under control defined in \eqref{17} for $\nu = 1, \alpha=1,\beta=1,\delta=2,\gamma=0.5,\eta = 1, \mu = 0.1 $ and $u_{0}(x) = \sin(\pi x)$.}
	\label{fig7}
\end{figure}
\begin{figure}[h]
	\centering
	% First image
	\begin{minipage}{0.68\textwidth}
		\centering
		\includegraphics[width=\linewidth]{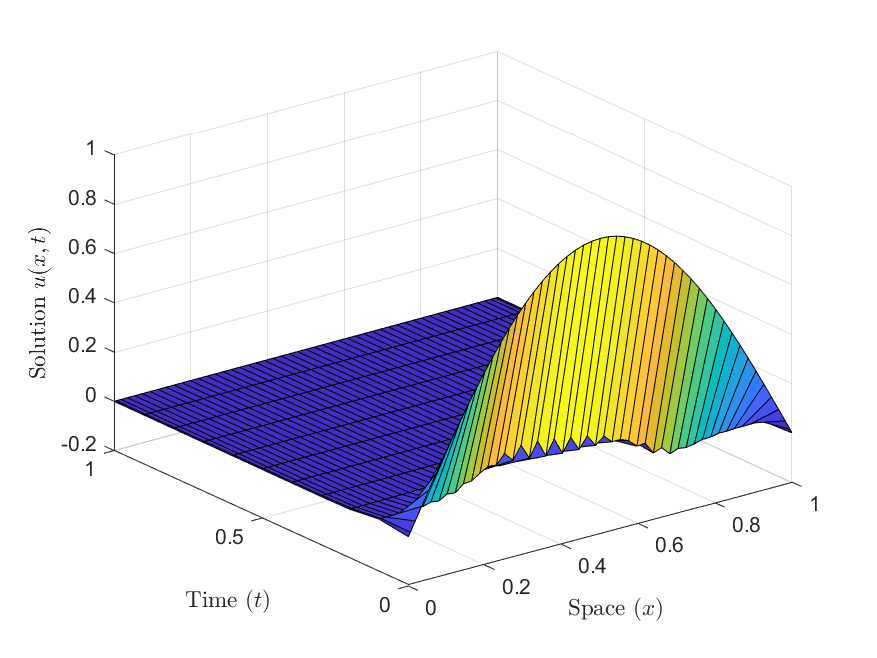}
	\end{minipage}\hfill
	% Second image
	\begin{minipage}{0.28\textwidth}
		\centering
		\includegraphics[width=\linewidth]{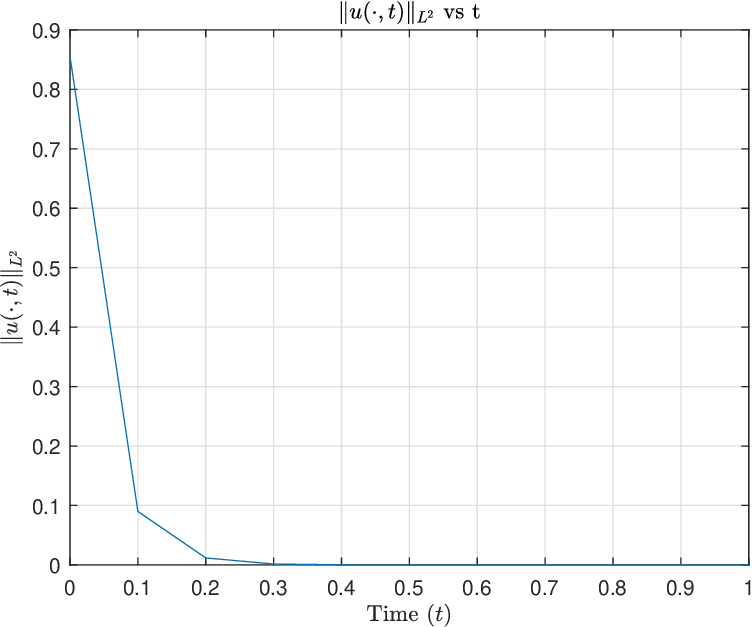} 
	\end{minipage}
	\caption{Time evolution of GKdVBH equation under control defined in \eqref{256} for $\nu = 1, \alpha=1,\beta=1,\delta=3,\gamma=0.5,\eta = 1, \mu = 0.1 $ and $u_{0}(x) = \sin(\pi x)$.}
	\label{fig8}
\end{figure}

\begin{figure}[h]
	
	\begin{minipage}{0.6\textwidth}
		\centering
		\includegraphics[width=\linewidth]{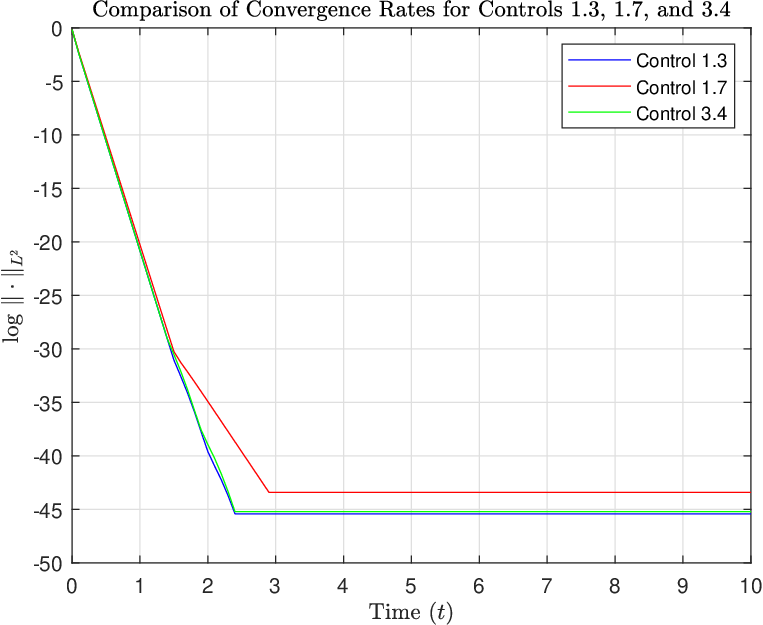} 
	\end{minipage}
	\caption{Time evolution of the log $\| u(\cdot, t) \|_{L^2}$ under the control \eqref{1.3}, \eqref{17} and \eqref{256} for $\delta=1$.}
	\label{fig9}
\end{figure}
\begin{figure}[h]
	
	\begin{minipage}{0.6\textwidth}
		\centering
		\includegraphics[width=\linewidth]{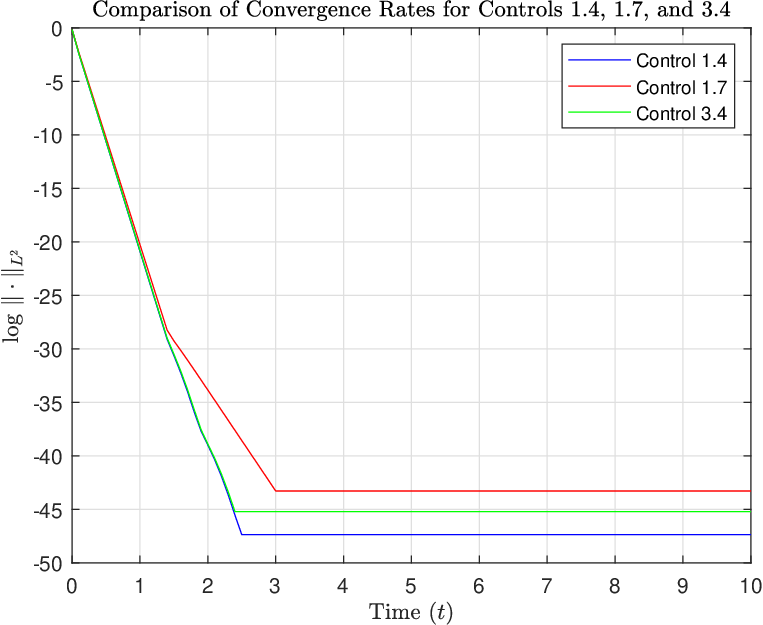} 
	\end{minipage}
	\caption{Time evolution of the log $\| u(\cdot, t) \|_{L^2}$ under the control \eqref{1p3}, \eqref{17} and \eqref{256} for $\delta=2$.}
	\label{fig10}
\end{figure}
\begin{figure}[h]
	
	\begin{minipage}{0.6\textwidth}
		\centering
		\includegraphics[width=\linewidth]{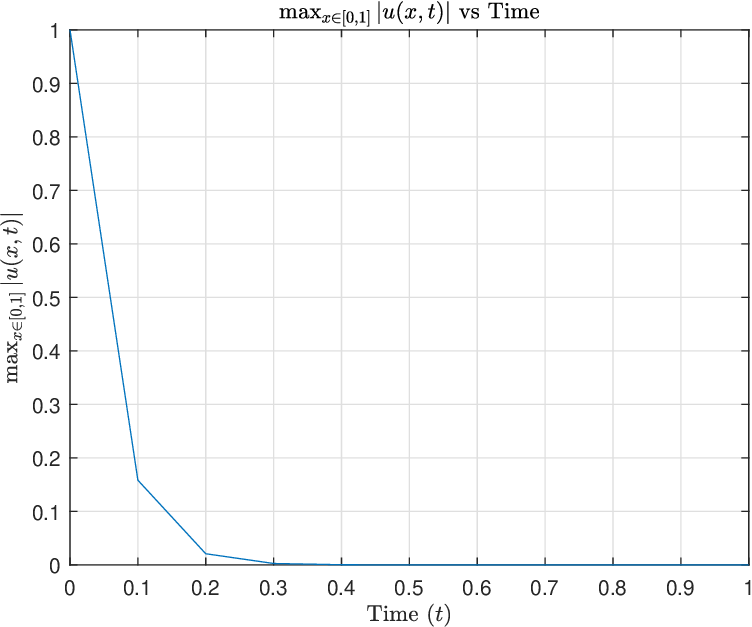} 
	\end{minipage}
	\caption{Time evolution of the $\max_{x \in [0,1]} |u(x,t)|$ under the control \eqref{256} for $\delta=1$.}
	\label{fig11}
\end{figure}

	\begin{appendix}
	\renewcommand{\thesection}{\Alph{section}}
	\numberwithin{equation}{section}
	\section{Some Useful Inequalities}\label{sec5}\setcounter{equation}{0} 
\begin{lemma}[Poincar\'e's inequality]\label{lem3.1}
	For any $w\in\C^1([0,1])$, we have 
	\begin{align}\label{eq0}
		\int_0^1w^2(x)dx\leq \left(\int_0^1w_x^2(x)dx\right).
	\end{align}
\end{lemma}

\begin{proof}
	For any $w\in\C^1[0,1]$, we know that 
	\begin{align*}
		w(x)=\int_0^xw_x(\zeta)d\zeta,
	\end{align*}
	so that 
	\begin{align*}
		w^2(x)=\left(\int_0^xw_x(\zeta)d\zeta\right)^2\leq\left(\int_0^1|w_x(\zeta)|^2d\zeta\right).
	\end{align*}
	Ingratiating the above inequality from $0$ to $1$, we have the  Poincar\'e inequality \eqref{eq0}.  By a density argument, the result \eqref{eq0} holds true for all $w\in\H^1(0,1)$. 
\end{proof} 
	
		\begin{lemma}[Agmon's inequality]\label{lem21}
		For any $w\in\C^1([0,1])$, the following inequality holds: 
		\begin{align}\label{21}
			\max_{x\in[0,1]}|w(x)|\leq \sqrt{2}\|w\|_{\L^2}^{1/2}\|w_x\|_{\L^2}^{1/2}.
		\end{align}
	\end{lemma}
	\begin{proof}
		An application of H\"older's inequality yields 
		\begin{align}\label{22}
			w^2(x)&=2\int_0^xw(\zeta)w_x(\zeta)d\zeta\leq 2\left(\int_0^1|w(\zeta)|^2d\zeta\right)^{1/2}\left(\int_0^1|w_x(\zeta)|^2d\zeta\right)^{1/2},
		\end{align}
		for all $x\in[0,1]$. 
		Therefore, from \eqref{22}, we infer 
		\begin{align}\label{24}
			\max\limits_{x\in[0,1]}|w(x)|^2\leq 2\|w\|_{\L^2}\|w_x\|_{\L^2},
		\end{align}
		and \eqref{21} follows. By a density argument, the result \eqref{21} holds true for all $w\in\H^1(0,1)$. 
	\end{proof}

	\begin{lemma}\label{lema3}
		For any $w\in\C^1([0,1])$, the following inequality holds: 
		\begin{align}\label{A5}
			\max_{x\in[0,1]}|w(x)|\leq(\delta+2)^{\frac{1}{\delta+2}}\|w\|_{\L^{2(\delta+1)}}^{\frac{\delta+1}{\delta+2}}\|w_x\|_{\L^2}^{\frac{1}{\delta+2}}.
		\end{align}
	\end{lemma}
	\begin{proof}
		An application of H\"older's inequality yields 
		\begin{align}\label{A6}
			w^{\delta+2}(x)&=\int_0^x\partial_{\zeta}w^{\delta+2}(\zeta)d\zeta=(\delta+2)\int_0^xw^{\delta+1}(\zeta)w_x(\zeta)d\zeta\nonumber\\&\leq (\delta+2)\left(\int_0^1|w(\zeta)|^{2(\delta+1)}d\zeta\right)^{1/2}\left(\int_0^1|w_x(\zeta)|^2d\zeta\right)^{1/2},
		\end{align}
		for all $x\in[0,1]$. 
		Therefore, from \eqref{A6}, we infer 
		\begin{align}\label{A7}
			\max\limits_{x\in[0,1]}|w(x)|^{\delta+2}\leq (\delta+2)\|w\|_{\L^{2(\delta+1)}}^{\delta+1}\|w_x\|_{\L^2},
		\end{align}
		and \eqref{A5} follows. By a density argument, the result \eqref{A5} holds true for all $w\in\H^1(0,1)$. 
	\end{proof}
	\end{appendix}

\end{document}